\documentclass[preprint]{article}

\usepackage{neurips_2026}
\usepackage{graphicx}
\newcommand{\new}[1]{{#1}}

\usepackage[utf8]{inputenc} 
\usepackage[T1]{fontenc}    
\usepackage{hyperref}       
\usepackage{url}            
\usepackage{booktabs}       
\usepackage{amsfonts,amsmath,amssymb,amsthm}       
\usepackage{nicefrac}       
\usepackage{microtype}      
\usepackage{xcolor}         

\newtheorem{theorem}{Theorem}

\newtheorem{proposition}{Proposition}
\newtheorem{corollary}{Corollary}

\newcommand{\ie}{i.e.}

\newcommand{\re}[1][\empty]{\mathbb{R}^{#1}}
\newcommand{\ip}[1][\cdot,\cdot]{\left\langle #1 \right\rangle}
\newcommand{\rank}{\mbox{rank}}

\newcommand{\Dset}{\mathcal{D}}        
\newcommand{\Kset}{\mathcal{K}}        
\newcommand{\DCset}{\mathcal{D}_C}     
\newcommand{\sign}{\mbox{sign}}

\newcommand{\dimn}{n}
\newcommand{\dimm}{m}
\newcommand{\n}{\dimn}
\newcommand{\m}{\dimm}
\newcommand{\PSp}{\re[\m\times \n]}
\newcommand{\loss}{f}
\newcommand{\lmin}{\loss^\star}
\newcommand{\x}{x}
\newcommand{\xs}{\x^*}
\newcommand{\y}{y}
\newcommand{\eps}{\varepsilon}

\newcommand{\NS}{\mathcal{N}}
\newcommand{\SSat}{\mathcal{S}}

\newcommand{\xk}{\x_{k}}
\newcommand{\xp}{\x_{k+1}}
\newcommand{\xm}{\x_{k-1}}
\newcommand{\yk}{\y_{k}}
\newcommand{\yp}{\y_{k+1}}
\newcommand{\Gk}{G_k}
\newcommand{\Bk}{B_k}
\newcommand{\Bm}{B_{k-1}}

\newcommand{\dLoss}[1][\cdot]{\nabla \!\loss(#1)}
\newcommand{\Ok}{O_{k}}

\newcommand{\Hk}{H_{k}}
\newcommand{\Hp}{H_{k+1}}
\newcommand{\Hm}{H_{k-1}}
\newcommand{\Hz}{H_{0}}
\newcommand{\Hi}[1]{H_{#1}}

\newcommand{\xktilde}{\tilde\x_k}
\newcommand{\xkbar}{\bar\x_k}
\newcommand{\Bktilde}{\tilde B_k}
\newcommand{\Hkbar}{\bar H_k}

\newcommand{\Pk}{P_k}
\newcommand{\Pmin}{p_{\min}}
\newcommand{\Pmax}{p_{\max}}

\newcommand{\R}{\mathbb{R}}

\newcommand{\norm}[1]{\left\|#1\right\|}
\newcommand{\inner}[2]{\left\langle #1,\,#2\right\rangle}
\newcommand{\Id}{I}

\title{Convergence guarantees for Muon: New parameter regimes and generalizations}

\author{%
  Arthur Castello B. de Oliveira \\
  Department of Electrical and Computer Engineering\\
  Northeastern University\\
  Boston, MA 02115 \\
  \texttt{a.castello@northeastern.edu} \\
  \And
  Dhruv D. Jatkar\\
  Department of Electrical and Computer Engineering\\
  Northeastern University\\
  Boston, MA 02115 \\
  \texttt{jatkar.d@northeastern.edu} \\
  \And
  Guilherme S. Vicinansa\\
  Department of Telecommunications and Control\\
  Escola Politecnica, University of S\~ao Paulo\\
  S\~ao Paulo, SP 05508-010 \\
  \texttt{guilherme.vicinansa@gmail.com} \\
  \And
  Eduardo D. Sontag \\
  Department of Electrical and Computer Engineering and Department of Bioengineering\\
  Northeastern University\\
  Boston, MA 02115 \\
  \texttt{e.sontag@northeastern.edu} \\
}

\begin{document}

\maketitle

\begin{abstract}
  In this paper, we establish the first \new{assymptotic} convergence guarantees for the Muon algorithm \new{through a more accurate proxy for the Newton-Schultz iteration than the typical matrix sign function}. We prove that, for appropriate choices of hyperparameters, the iterates satisfy $\lim_{k\to\infty}\|\nabla f(x_k)\|=0$, and, under a global Polyak-\L{}ojasiewicz condition, that the {sequence of} function values converges linearly. The key insight is that the regularization, implicit in Muon's Newton-Schulz implementation, induces a bounded preconditioner, exposing Muon as a \emph{preconditioned Polyak heavy-ball} method and enabling a classical Lyapunov analysis. This observation naturally motivates applying the same preconditioning structure to the Nesterov gradient evaluation. We formalize this idea by introducing \emph{Muesterov}, a Nesterov-based variant of Muon, and prove that it enjoys the same convergence guarantees, extending the theoretical framework beyond the heavy-ball setting. Numerical experiments on a scalar cross-entropy problem corroborate the theory and illuminate the joint role of the learning rate and the Newton-Schulz regularizer in controlling convergence.  Preliminary numerical simulations training the nanoGPT dataset provide intuition regarding the relevance of the observations in this paper to practical applications.
\end{abstract}



\section{Introduction}

Training large transformer models has driven sustained interest in
optimizers that go beyond the element-wise second-moment 
estimates of Adam~\citep{kingma2014adam} and
AdamW~\citep{loshchilov2018decoupled}.
A number of approaches exploit richer structure in the gradient or its
second-order statistics. Methods such as K-FAC~\citep{martens2015optimizing}
and Sophia~\citep{liu2023sophia} incorporate curvature information, while
algorithms such as Shampoo~\citep{gupta2018shampoo} and
SOAP~\citep{vyas2024soap} leverage matrix-valued preconditioning and
second-moment structure to improve convergence. These methods can yield
substantial gains, but typically require maintaining and manipulating
additional matrix statistics.

A computationally cheaper alternative is to normalise the gradient matrix
directly. Bernstein and Newhouse~\citep{Bernstein2024} formalise this idea
through \emph{modular duality}, showing that for any layer-wise norm, the
steepest descent direction is given by its dual map. In the case of the
spectral (Schatten-$\infty$) norm, this dual corresponds to the
\emph{matrix sign function}~\citep{Higham2008},
$\mathrm{sign}(B) = UV^\top$, where $B = U\Sigma V^\top$ is an SVD.
This can be viewed as a matrix analogue of the element-wise sign used by
methods such as signSGD~\citep{bernstein2018signsgd} and
Lion~\citep{chen2023symbolic}.

\emph{Muon}~\citep{jordan2024muon} turns this principle into a practical
algorithm by approximating the matrix sign via a small number of
Newton-Schulz iterations, which can be implemented efficiently in
low precision and parallelised on modern hardware. Subsequent work has
studied improved matrix-sign iterations and their numerical
properties~\citep{amsel2025polar}, as well as the empirical performance
of Muon in large-scale language model training~\citep{liu2025muon}.
Combined with momentum, this orthogonalisation step leads to updates that
can be interpreted as a form of preconditioned momentum iteration, an
observation made in several recent theoretical
analyses~\citep{li2025note,shen2025convergence,sato2025convergence,chang2025convergence}.

Despite this progress, providing a classical convergence theory for Muon
remains challenging. A central difficulty is that the exact matrix sign
operator produces updates of essentially constant magnitude: for a
rank-$r$ matrix $B$, one has $\|\mathrm{sign}(B)\|_F = \sqrt{r}$,
so the induced preconditioner is not uniformly bounded. This prevents
direct application of standard Lyapunov arguments for preconditioned
methods. As a result, existing analyses typically establish
finite-time or asymptotic \emph{stationarity} guarantees -- often in
terms of averaged or expected gradient norms -- under smoothness and
stochastic assumptions~\citep{li2025note,shen2025convergence,sato2025convergence,kim2026convergence},
with improved rates obtainable under variance
reduction~\citep{chang2025convergence}. These results, while valuable,
do not yield classical full-sequence convergence guarantees for the
iterates themselves.

In this paper, we establish the first asymptotic convergence guarantees for a soft-sign variation of the classical formulation of Muon. Our key observation is that the regularization implicit in this variation of Muon’s Newton-Schulz implementation induces a uniformly bounded preconditioner, exposing it as a preconditioned Polyak heavy-ball method and enabling a classical Lyapunov analysis. Under suitable hyperparameter choices, we prove that the iterates satisfy $\lim_{k\to\infty}\|\nabla f(x_k)\|=0$ and, under a global Polyak--Łojasiewicz condition, that the function values converge linearly.

Motivated by this interpretation and the known advantages of Nesterov over Polyak heavy-ball \citep{lessard2016analysis,shi2021understanding}, we additionally introduce Muesterov, a Nesterov-based variant of Muon obtained by applying the same normalization structure to extrapolated gradients. We prove that Muesterov satisfies analogous convergence guarantees, and compare it to Muon on scalar synthetic simulations.

Finally, we present synthetic and transformer-based training experiments illustrating how the learning rate and the Newton-Schulz regularizer jointly control the transition between oscillatory sign-like dynamics and stable convergence. In particular, the fixed-batch nanoGPT experiments illustrate that the same regularization mechanisms identified by the theory and synthetic simulations remain qualitatively visible in transformer training.

All proofs are provided in the technical appendix, and all code used to generate the numerical results is available in \cite{why_muon_works_2026}. 

\section{Problem Setup and Preliminary Analysis}
\label{sec:ProbForm}


The set of natural and real numbers are denoted $\mathbb{N}$
and $\mathbb{R}$ respectively. Let $\|\cdot\|$ and
$\langle\cdot,\cdot\rangle$ denote the Frobenius norm and inner product
if taken over matrices, and the $\ell_2$ norm and inner product if taken
over vectors. Given $m,n\in\mathbb{N}$ and a differentiable function
$f:\mathbb{R}^{m\times n}\to\mathbb{R}$, let
$f^*:=\min_{x\in\mathbb{R}^{m\times n}}f(x)$ denote its minimum value,
assumed to be well defined and finite throughout. 

A function $\loss$ is said to be \emph{proper} if the pre-images of compact sets are compact. Furthermore, a function $\loss$ is said to be \emph{$L$-smooth} for some $L>0$ if for all $x,y\in\PSp$ it holds that
\begin{equation}
    \label{eq:Lsmoothness-def}
    \|\dLoss[x]-\dLoss[y]\|\leq L\|x-y\|.
\end{equation}
Finally, a function $\loss$ is said to satisfy a \emph{global Polyak-\L{}ojasiewicz inequality} (P\L{}I) with constant $\mu>0$ if for every $x\in\re[m\times n]$ it holds that
\begin{equation}
    \label{eq:PLI-def}
    \|\nabla\!\loss(x)\|^2\geq\mu(\loss(\x)-\loss^*). 
\end{equation}
For any matrix $x\in\mathbb{R}^{m\times n}$, let $x = U\Sigma V^\top$ 
be a valid singular value decomposition (SVD) of $x$. Then the 
\emph{matrix sign function}~\citep{Higham2008} is given by $\mathrm{sign}(x) := UV^\top$, and is well defined up to the choice of singular vectors corresponding to zero singular values.   

\subsection{Equivalent Muon formulations}
We now introduce three different formulations for Muon: the practical implementation as described by \cite{jordan2024muon}; an idealized sign-based model typically considered in the literature; and a regularized analytical proxy to be used for convergence analysis. We will discuss the equivalence between these three algorithms and why the distinctions matter.

For $n\ge m\in\mathbb{N}$, let $\loss:\PSp\to\re$ be an $L$-smooth function that satisfies a global Polyak-\L{}ojasiewicz inequality with constant $\mu>0$. Furthermore, let $\loss$ be proper and bounded below, with minimum value $\loss^*:=\loss(\xs)$ attained at some finite point $\xs\in\PSp$. Consider, then, the following optimization problem.
\begin{equation}
    \label{eq:optprob-def}
    \underset{\x\in\PSp}{\mathrm{minimize}}\qquad f(x).
\end{equation}
Many methods exist for solving \eqref{eq:optprob-def}, however Muon \citep{jordan2024muon} has rapidly gained widespread popularity for the training of the hidden-layers of transformer models. Muon (henceforth called \emph{canonical Muon}, or simply \emph{Muon}) proposes finding the solution of \eqref{eq:optprob-def} by the following iteration
\begin{subequations}
    \label{eq:Muon-Orig}
    \begin{align}
        \Bk &= \beta \Bm+\dLoss[\xm]\\
        \Ok &= \NS(\Bk)\\
        \xk&=\xm-\alpha\Ok,
    \end{align}
\end{subequations}
where $\NS(x)$ denotes the application of a Newton–Schulz iteration, empirically tuned for Muon applications, to the normalized matrix $x/(\|x\|_F+\delta)$. 

Prior work \citep{jordan2024muon,Bernstein2024,amsel2025polar} shows that Muon's Newton-Schulz algorithm aims at approximating the matrix sign function (as defined in \citep{Higham2008}). In other words, if $U_k\Sigma_k V_k^\top$ is a valid singular value decomposition of $\Bk$, then $\Ok\approx U_kV_k^\top=:\sign(\Bk)$. From this observation, an idealized version of Muon (henceforth called \emph{ideal Muon}) can be written as 
\begin{subequations}
    \label{eq:Muon-Ideal}
    \begin{align}
        \Hk &= \beta \Hm+(1-\beta)\dLoss[\xm]\\
        \xk&=\xm-\alpha \sign(\Hk).
    \end{align}
\end{subequations}
Furthermore, the ideal Muon is equivalent to the canonical Muon in the following sense.
\begin{proposition}
    \label{prop:Muon-OrigIdeal-Equivalence}
    Let $\NS(\Bk)=\sign(\Bk)$. Also, for any $x_0\in\PSp$ let $\{\xktilde,\Bktilde\}$ and $\{\xkbar,\Hkbar\}$ be the sequences generated by \eqref{eq:Muon-Orig} and \eqref{eq:Muon-Ideal} respectively, each initialized at $\{x_0, B_0\}$ and $\{x_0, \Hz\}$ with $B_0=(1-\beta)^{-1}\Hz$. Then, for every $k\in\mathbb{N}$ it holds that $\xktilde=\xkbar$ and $\Bktilde=(1-\beta)^{-1}\Hkbar$.
\end{proposition}

The proposition shows that the two different formulations induce identical trajectories in the parameter space, and up to a rescaling of the momentum. 
%
%
The $(1-\beta)$ factor is introduced to match the standard heavy-ball form, in which the gradient appears with weight $(1-\beta)$. This scaling does not affect Proposition~\ref{prop:Muon-OrigIdeal-Equivalence}, as it only rescales the momentum variable while leaving the induced updates unchanged. Crucially, it places \eqref{eq:Muon-Ideal} in the form of a preconditioned Polyak heavy-ball iteration, which will be exploited in the next section.

%
Note that Ideal Muon does not converge asymptotically under constant $\alpha$ and $\beta$, since $\|\sign(\Hk)\|_F=\sqrt{\rank(\Hk)}$ prevents vanishing step sizes near critical points. Existing analyses therefore establish $\varepsilon$-stationarity or averaged-gradient guarantees rather than pointwise convergence \citep{shen2025convergence,li2025note}. To address the issue of absence of vanishing step size near critical points, we introduce the soft-sign normalization.
\begin{equation}
    \label{eq:SoftSignDef}
    \SSat(x,\eps):=\frac{x}{\sqrt{x^2+\eps^2}}
\end{equation}
for any $\eps>0$. The matrix-valued soft sign $\SSat:\PSp\to\PSp$ is, then, simply the scalar soft sign applied to the singular values of the input matrix. That is, let $\x=U\Sigma V^\top$ be a valid SVD of a matrix $\x\in\PSp$, then
\begin{equation}
    \SSat(\x,\eps):=U\SSat(\Sigma,\eps)V^\top = (\x \x^\top+\eps^2 I)^{-1/2}\x,
\end{equation}
which is well-defined for all $\x \in \PSp$. This allows the formulation of the following iteration (henceforth called \emph{Soft Muon})
\begin{subequations}
    \label{eq:Muon-Soft}
    \begin{align}
        \Hk &= \beta \Hm+(1-\beta)\dLoss[\xm]\\
        \xk&=\xm-\alpha(\Hk\Hk^\top+\eps^2 I)^{-1/2}\Hk,
    \end{align}
\end{subequations}
which can also be shown to be equivalent to Muon in the following sense.

\begin{proposition}
    \label{prop:Muon-OrigSoft-Equiv}
    Let $\NS(\Bk)=\SSat(\Bk,\eps)$ for some $\eps>0$. Also, for any $x_0\in\PSp$ let $\{\xktilde,\Bktilde\}$ and $\{\xkbar,\Hkbar\}$ be the sequences generated by \eqref{eq:Muon-Orig} and \eqref{eq:Muon-Soft} respectively, with each initialized at $\{x_0, B_0\}$ and $\{x_0, \Hz\}$ with $B_0=(1-\beta)^{-1}\Hz$. Further assume \eqref{eq:Muon-Soft} generates its sequence with constant $\eps(1-\beta)$ instead of $\eps$. Then, for every $k\in\mathbb{N}$ it holds that $\xktilde=\xkbar$ and $\tilde \Bk=(1-\beta)^{-1}\Hkbar$.
\end{proposition}
%

Note that neither Proposition \ref{prop:Muon-OrigIdeal-Equivalence} nor \ref{prop:Muon-OrigSoft-Equiv} assert that the Newton-Schulz polynomial is exactly equal to either the matrix sign or the soft sign. Rather, they show that once a particular surrogate for the Newton-Schulz map is chosen, the resulting canonical formulation can be rewritten exactly in the corresponding heavy-ball form.


Unlike ideal Muon, the soft-sign formulation permits vanishing updates near critical points, enabling asymptotic convergence analysis, as we will show later in the paper.

\subsection{Why so many Muons?}

\begin{figure}
    \centering
    \includegraphics[width=0.2\linewidth]{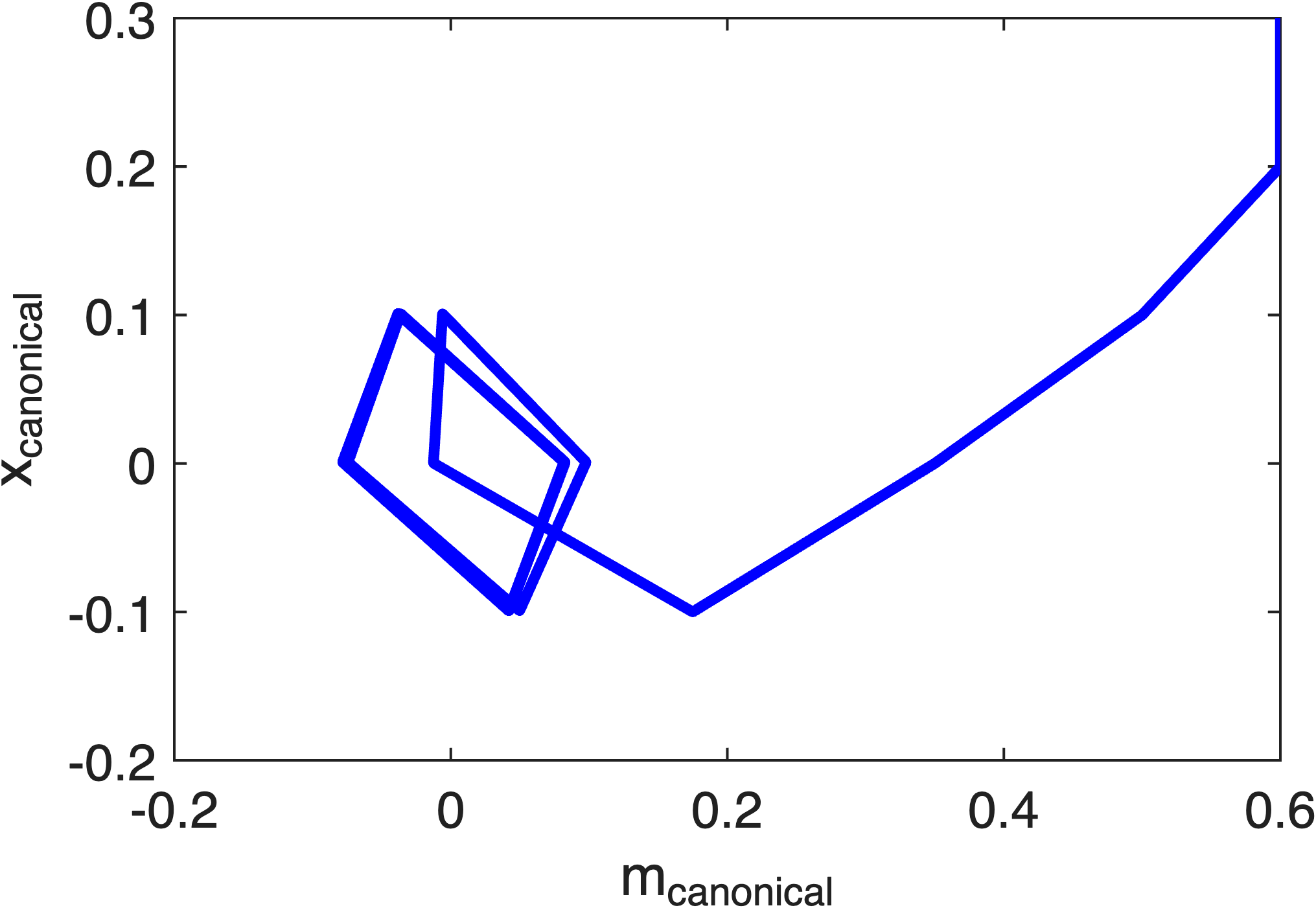}
    \includegraphics[width=0.2\linewidth]{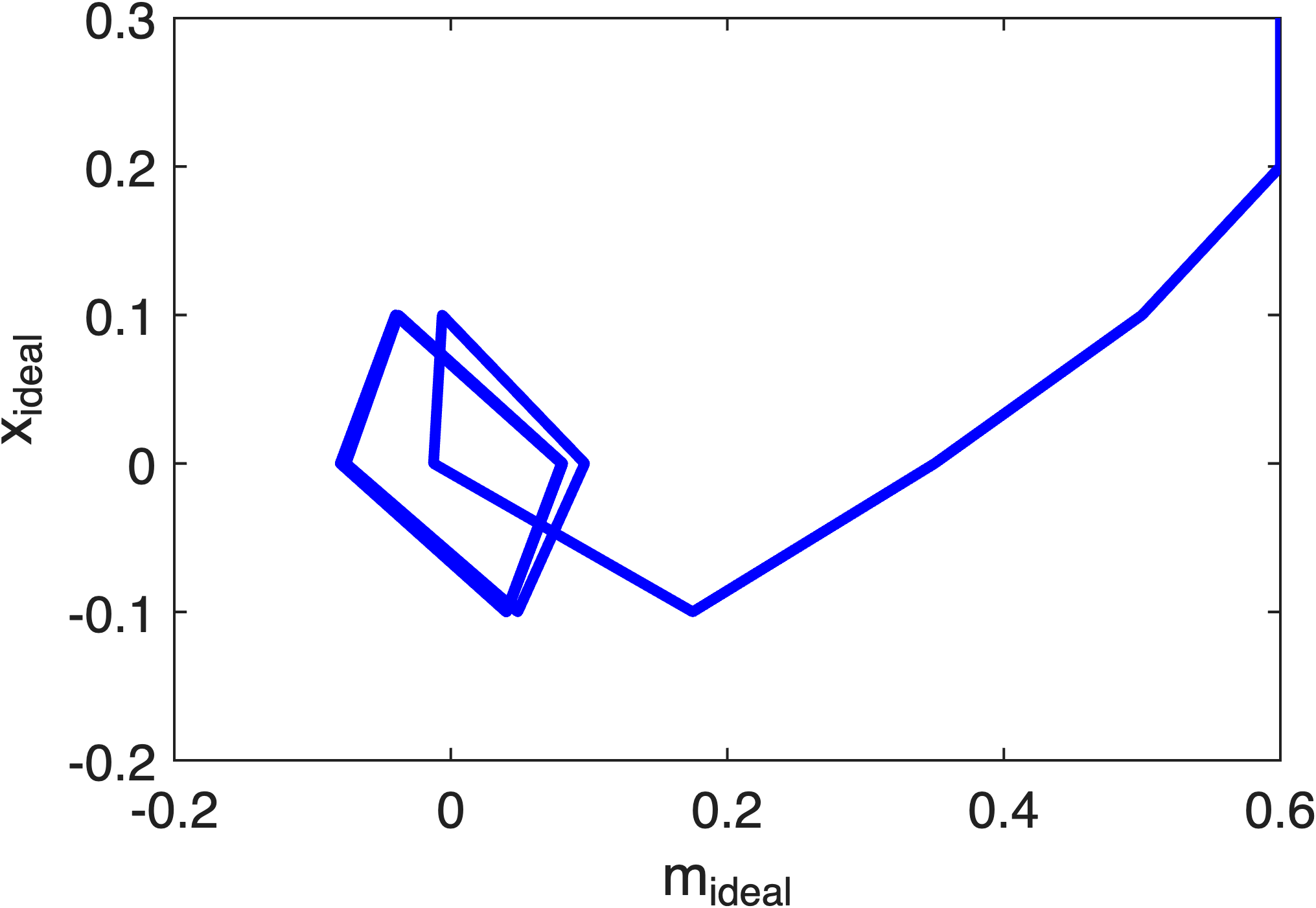}
    \includegraphics[width=0.2\linewidth]{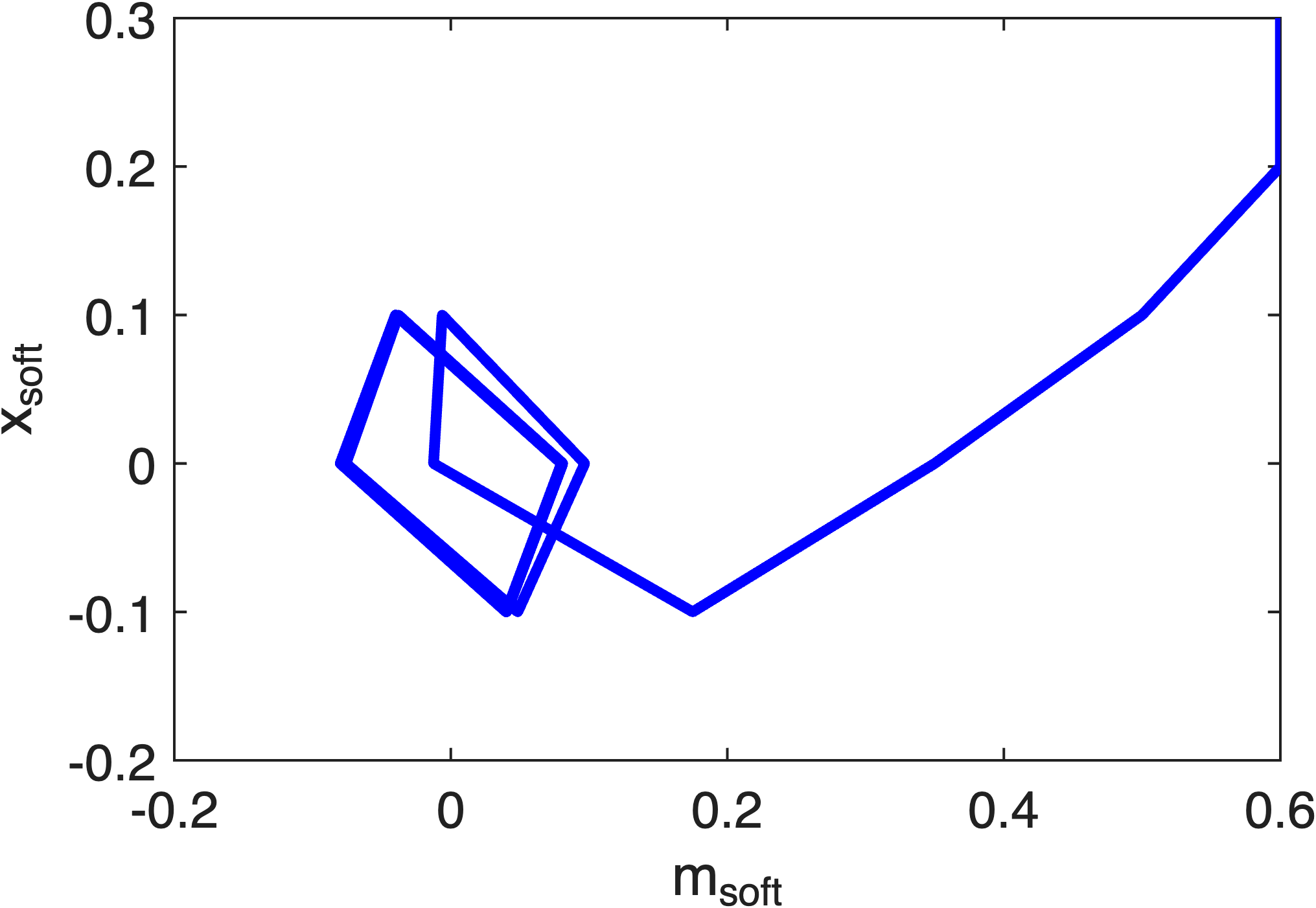}
    \includegraphics[width=0.2\linewidth]{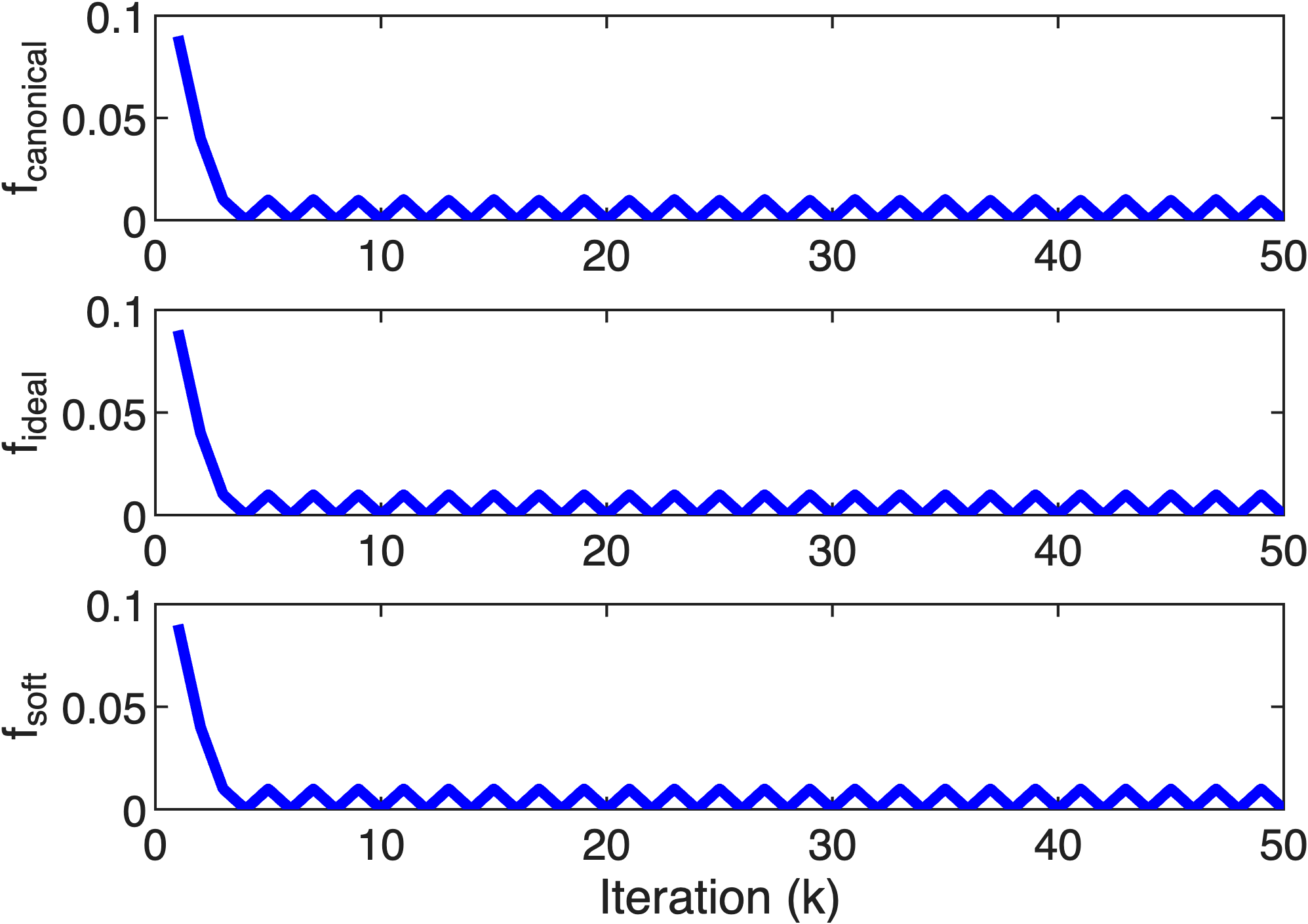}\\
    \includegraphics[width=0.2\linewidth]{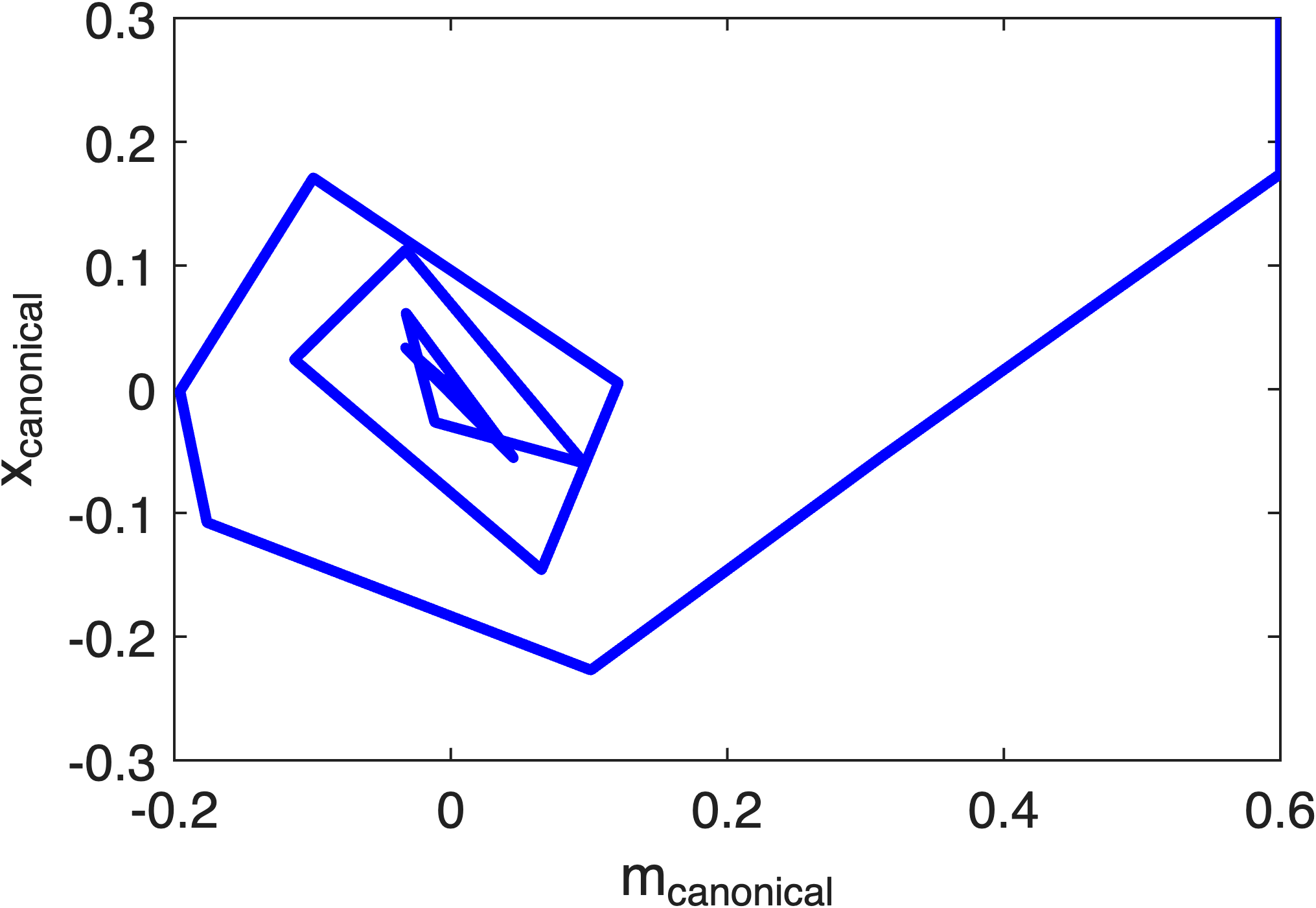}
    \includegraphics[width=0.2\linewidth]{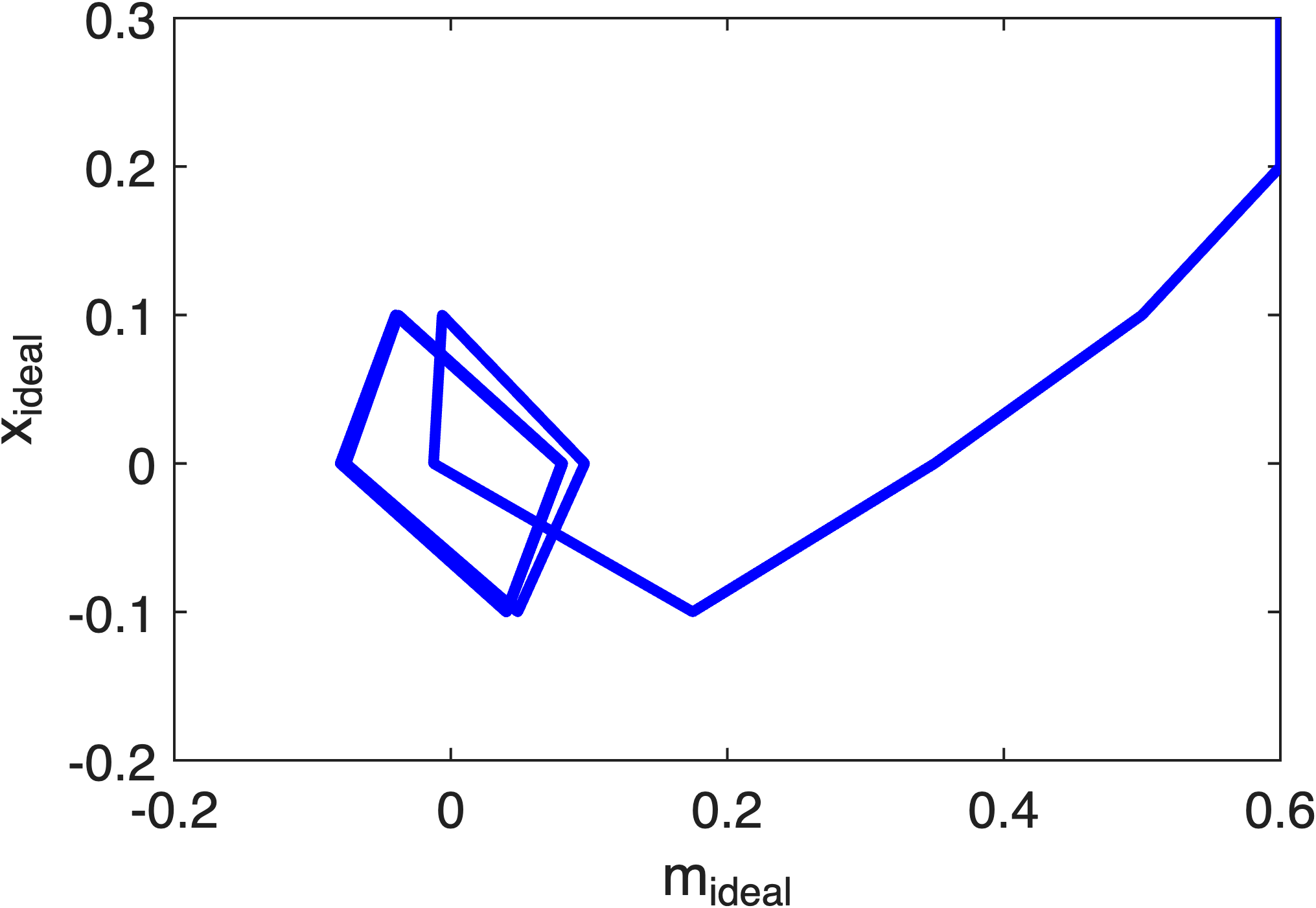}
    \includegraphics[width=0.2\linewidth]{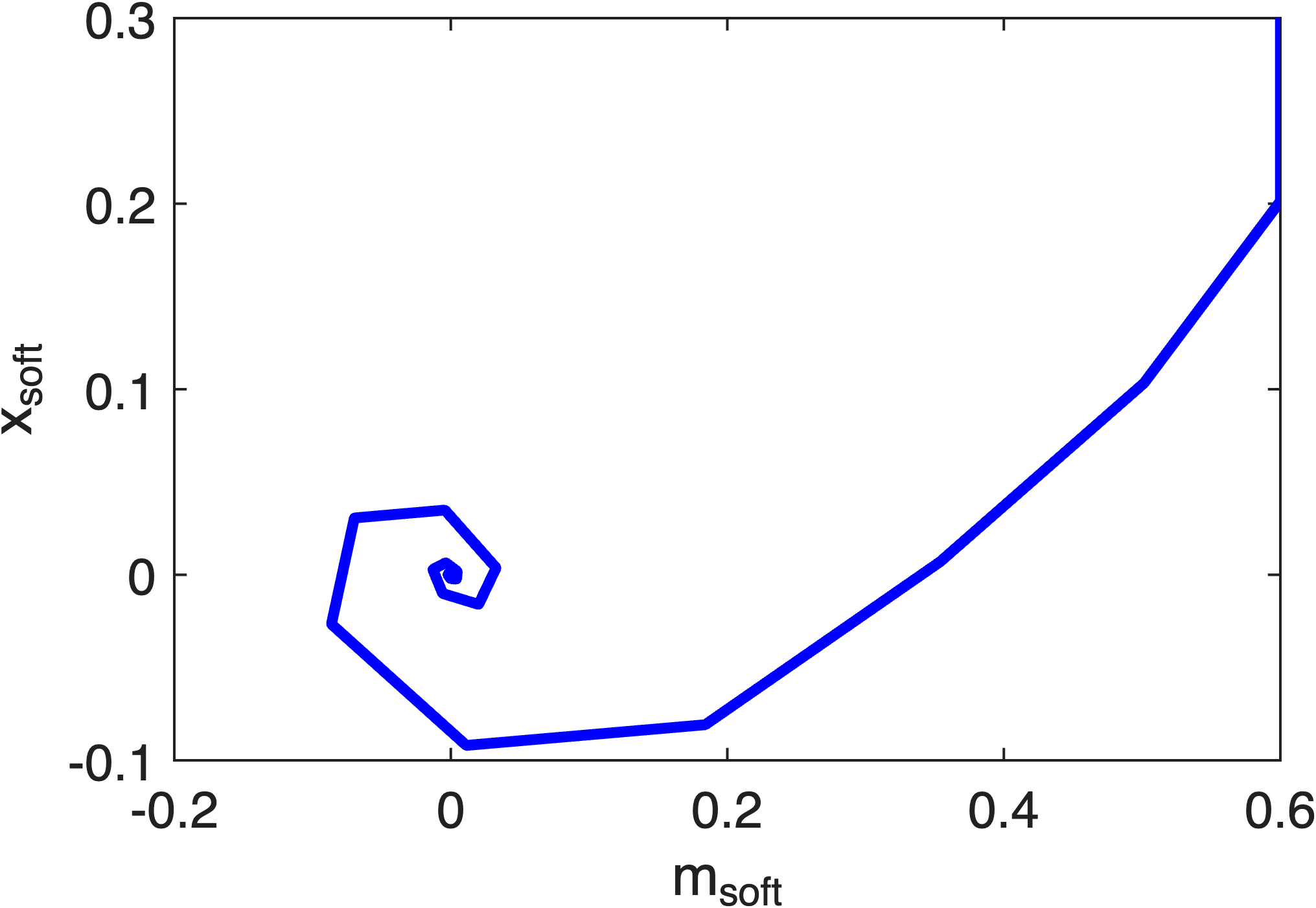}
    \includegraphics[width=0.2\linewidth]{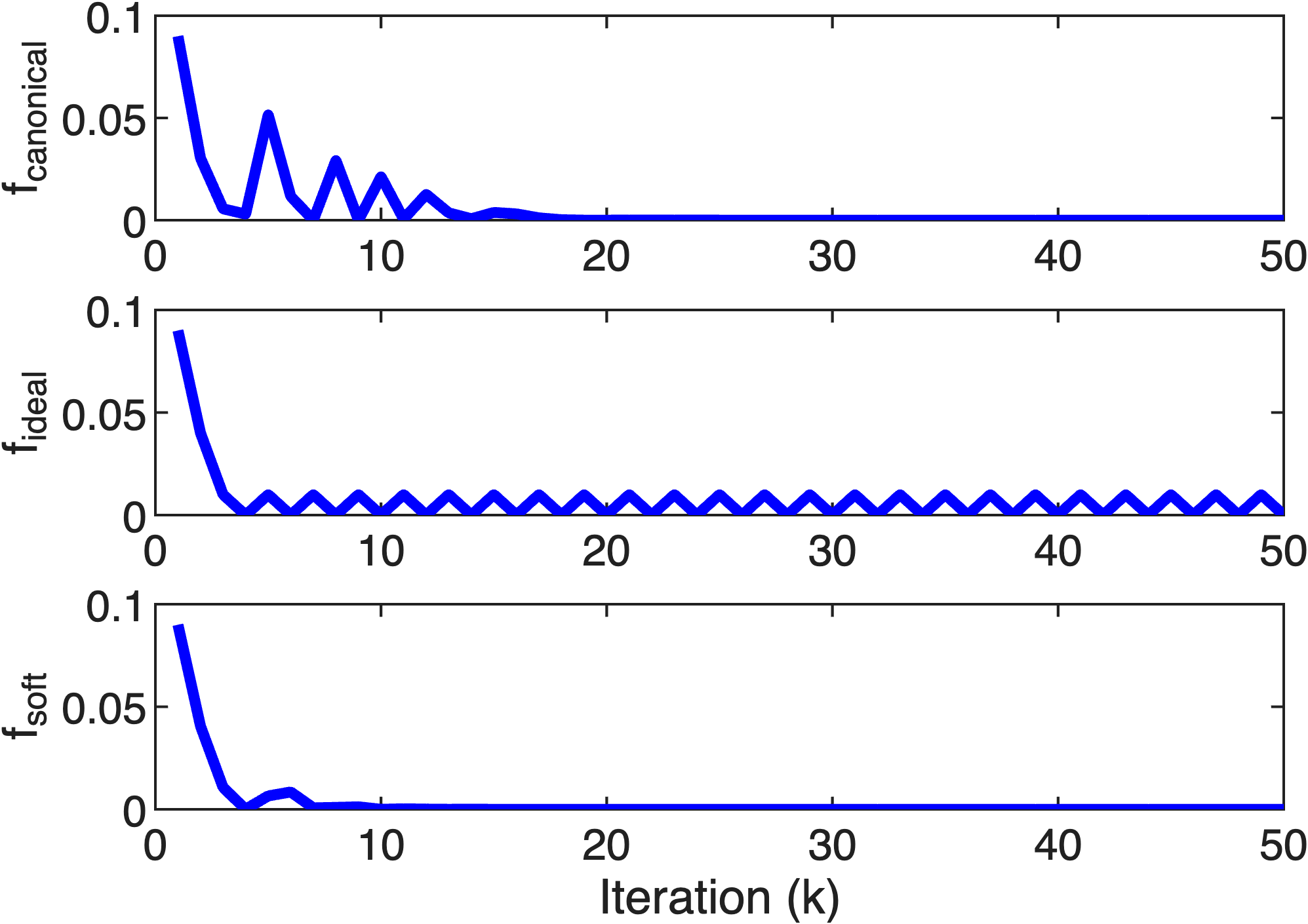}\\
    ~~~~(a)\qquad\qquad\qquad~~~~(b)\quad\qquad\qquad\qquad(c)\qquad\qquad\qquad~~~(d)
    \caption{Simulations for $f(x)=x^2$ with $\alpha=0.1$ and $\beta=0.5$. Columns (a)-(c) show the phase portrait ($x_k\times H_k$) for Canonical Muon \eqref{eq:Muon-Orig}, Ideal Muon \eqref{eq:Muon-Ideal}, and Soft Muon \eqref{eq:Muon-Soft} respectively; (d) shows the loss across iterations. The first row uses $\eps=10^{-7}$ and $\delta = k_{\delta\eps} 10^{-7} = 484.876 10^{-7}$, where $\delta$ denotes the regularizer inside the Newton-Schulz algorithm. To generate the second row, the values of $\delta$ and $\eps$ were increased until Canonical Muon and Soft Muon respectively started presenting a converging trajectory. Ideal Muon has no such regularizer and instead always settles into a periodic orbit around $x^*=0$.}
    \label{fig:QuadraticExample}
\end{figure}

Propositions~\ref{prop:Muon-OrigIdeal-Equivalence} and \ref{prop:Muon-OrigSoft-Equiv} establish conditions in which Canonical Muon \eqref{eq:Muon-Orig}, Ideal Muon \eqref{eq:Muon-Ideal}, and Soft Muon \eqref{eq:Muon-Soft} generate identical parameter trajectories up to a rescaling of the momentum variable. Figure~\ref{fig:QuadraticExample} illustrates how these formulations begin to differ once regularization becomes non-negligible.
The first row corresponds to a weakly regularized regime, in which all three methods behave nearly identically and fail to converge asymptotically. In contrast, the second row shows that increasing $\delta$ and $\eps$ sufficiently induces vanishing step sizes near the optimum: Canonical and Soft Muon now converge toward $x^\star=0$, while Ideal Muon continues to orbit the minimizer. Thus, the three formulations differ only through the effect of regularization -- when the regularization is negligible they coincide, while sufficiently strong regularization restores asymptotic convergence.

Although Canonical and Soft Muon exhibit the same qualitative transition, their trajectories may differ quantitatively for small values of $\delta$ and $\eps$, reflecting the higher-order structure of the Newton-Schulz polynomial. Throughout the paper, we align the local behavior of the two formulations near the origin by matching their scalar first-order expansions, yielding the relation $\delta \approx 484.876\,\eps=:k_{\delta\eps}\,\eps$ used in the simulations.
Finally, both \eqref{eq:Muon-Ideal} and \eqref{eq:Muon-Soft} can be interpreted as preconditioned variants of the Polyak heavy-ball method. The next section develops a convergence theory for preconditioned heavy-ball iterations and specializes it to Soft Muon.

\section{Muon and the preconditioned Polyak heavy-ball}
\label{sec:Muon}

The soft-sign formulation exposes Muon as a preconditioned heavy-ball method, allowing classical convergence analysis through the induced regularized preconditioner. The \emph{generic preconditioned heavy-ball iteration} can be written as
\begin{subequations}\label{eq:PrecHeavyBall}
\begin{align}
  \Hk &= \beta \Hm + (1-\beta)\dLoss[\xm], \\
  \xk &= \xm - \alpha \Pk^{-1} \Hk,
\end{align}
\end{subequations}
where $P_k \in \mathbb{R}^{m \times m}$ is a positive definite preconditioning matrix, and $P_k^{-1} H_k$ defines the descent direction. For this formulation, the following result can be stated.
\begin{theorem}\label{thm:PrecHeavyBall}
Consider the preconditioned heavy-ball iteration presented in~\eqref{eq:PrecHeavyBall} for solving \eqref{eq:optprob-def}, and assume
the preconditioning matrix $\Pk$ is uniformly bounded  -- \ie~there exist $\Pmax>\Pmin>0$ such that $\Pmin I\preceq \Pk\preceq \Pmax I$ for all $k\in\mathbb{N}$.
Then there exist a constant $\beta^\star\in(0,1)$ and a positive function
$\alpha^\star:[0,\beta^\star)\to\R_{>0}$ such that for every $\beta\in[0,\beta^\star)$
and every $\alpha\in(0,\alpha^\star(\beta))$, it holds that
\begin{equation}
  \lim_{k\to\infty}\norm{\dLoss[\xk]} = 0, \label{eq:convergence}
\end{equation}
that is, the method converges to a first-order stationary point in the classical sense.
\end{theorem}


Furthermore, linear convergence follows naturally from this and the P\L{} condition, as shown next.


\begin{corollary}\label{cor:MuonSoft-LinearConvergence}
    Let $f$ satisfy a global Polyak--\L{}ojasiewicz inequality with constant $\mu>0$, and let $P_k$ be uniformly bounded.
    Then, for every $\beta\in[0,\beta^\star)$ and every
    $\alpha\in(0,\alpha^\star(\beta))$ (where $\beta^\star$ and $\alpha^\star(\cdot)$ are as in Theorem \ref{thm:PrecHeavyBall}), there exist constants
    $C>0$ and $\rho\in(0,1)$ such that
    \begin{equation}
        f(x_k)-f^\star \leq C\rho^k,
        \qquad \forall k\in\mathbb N.
    \end{equation}
\end{corollary}

Theorem ~\ref{thm:PrecHeavyBall} and Corollary ~\ref{cor:MuonSoft-LinearConvergence} provide a general convergence framework for preconditioned heavy-ball methods. In particular, they show that classical convergence guarantees hold whenever the preconditioning matrices remain uniformly bounded above and below. 

In the context of Muon, this condition is not immediate. While the regularization in the soft-sign formulation induces a natural lower bound on the preconditioner, a global upper bound does not follow directly from the assumptions on $\loss$. Establishing such a bound therefore becomes the key technical step required to apply the general theory.

We will next show that the Soft Muon iteration \eqref{eq:Muon-Soft} generates bounded trajectories under suitable conditions, which in turn implies uniform boundedness of the associated preconditioner. This allows us to invoke Theorem~\ref{thm:PrecHeavyBall} and conclude convergence of the method.

\subsection{Why soft Muon works}

The convergence results for preconditioned heavy-ball methods clarify why the soft-sign formulation is the right proxy for Muon's Newton-Schulz iteration from an analysis perspective: the regularization parameter $\eps$ provides a uniform lower bound on the induced preconditioner.

The lower bound alone is not sufficient to apply Theorem~\ref{thm:PrecHeavyBall}, since a uniform upper bound on $P_k$ would also require a uniform bound on $\Hk$. Such a bound is not available globally, because $P_k=(\Hk\Hk^\top+\eps^2 I)^{1/2}$ grows with the momentum variable. The main purpose of this subsection is therefore to prove boundedness of the generated trajectory, from which the required upper bound follows.

For that purpose, a key property of $d_{\eps}(\Hk) := (\Hk \Hk^\top + \eps^2\Id)^{-1/2}\Hk$ is that $\|d_{\eps}(H)\|<\sqrt{\dimm}$ for every $H\in\R^{\dimm\times\dimn}$, so

\begin{equation}
  \|\xk - \xm\| = \alpha\|d_{\eps}(\Hk)\| < \sqrt{\dimm}\alpha \label{eq:stepbound}
\end{equation}

holds unconditionally for any $\beta\in[0,1)$. This bound is independent of both $\|\Hk\|$ and $\|\nabla \loss(\xk)\|$, and is the mechanism that prevents arbitrarily large one-step excursions, allowing us to prove that for any initialization $x_0\in\PSp$, the sequence generated by \eqref{eq:Muon-Soft} is bounded, circumventing the upper-bound issue.

\begin{theorem}\label{thm:muon-compact}
  Let $\{\xk,\Hk\}$ be the sequence generated by~\eqref{eq:Muon-Soft} from the
  initial condition $\{x_0, \dLoss[x_0]\}$. Assume $f$ is proper, $L$-smooth,
  has a unique global minimizer $x^\star$, and satisfies a global
  Polyak--\L ojasiewicz inequality with constant $\mu>0$. Then there exists a positive function
  $\alpha^\star:[0,1)\to\R_{>0}$ such that for every
  $\beta\in[0,1)$ and for every $\alpha\in(0,\alpha^\star(\beta))$
  the sequences $\{\xk\}$ and $\{\Hk\}$ are bounded.
\end{theorem}

Once boundedness of the iterates is established, the preconditioner $\Pk=(\Hk\Hk^\top+\eps^2I)^{1/2}$ is uniformly bounded along any trajectory. Hence the hypothesis of Theorem \ref{thm:PrecHeavyBall} hold, yielding following convergence result.
 
\begin{corollary}\label{cor:muon-convergence}
  Consider the Soft Muon~\eqref{eq:Muon-Soft}, then there exists $\bar\alpha(\beta)>0$ such that for every $\beta\in[0,\beta^\star)$ and
  for every
  $\alpha\in(0,\bar\alpha(\beta))$ the sequences $\{\xk\}$ and $\{\Hk\}$ are
  bounded and $\lim_{k\to\infty}\|\dLoss[\xk]\|=0$, where
  $\beta^\star=1/\sqrt{2}$ is the threshold from Theorem \ref{thm:PrecHeavyBall}.
\end{corollary}





Corollary~\ref{cor:MuonSoft-LinearConvergence} establishes linear convergence of Soft Muon under a global P\L{} inequality, provided the learning rate and momentum parameter satisfy the bounds of Theorem~\ref{thm:PrecHeavyBall}. Together with Theorem~\ref{thm:muon-compact}, this shows that the regularization induced by the soft-sign formulation restores classical convergence guarantees despite the normalization-based update structure.

Note that the constraint $\beta<1/\sqrt{2}$ does not reflect the empirical use of Muon, where $\beta$ is typically picked at around $0.95$. The constraint comes from the general result in Theorem \ref{thm:PrecHeavyBall}, where convergence needs to be guaranteed for \emph{any} bounded preconditioner, and does not mean that $\beta>1/\sqrt{2}$ will necessarily break convergence for Muon specifically. We believe that the specific structure of the Newton-Schulz iteration can be leveraged to tighten this constraint, and will investigate this gap in future works.


Having established this connection between normalization-based updates and preconditioned heavy-ball dynamics, we next investigate whether the same perspective can be extended to Nesterov-type acceleration schemes.



\section{A Nesterov-based Muon-type algorithm}
\label{sec:Muesterov}

Nesterov acceleration naturally motivates investigating whether the normalization structure identified for Muon can be extended beyond heavy-ball dynamics,
and whether similar convergence guarantees can be established.
In light of this question, consider the preconditioned Nesterov algorithm below.
\begin{subequations}
    \label{eq:PrecNesterov-def}
    \begin{align}
        \yk &= \xk + \beta (\xk - \xm), \\
        \xp &= \yk - \alpha \Pk^{-1}\dLoss[\yk],
    \end{align}
\end{subequations}
where $\alpha>0$, $\beta\in[0,1)$ are the learning rate and momentum parameter, respectively, $y_k$ is the momentum-extrapolated point at which the gradient is evaluated, and$P_k^{-1}\nabla f(y_k)$ defines the preconditioned descent direction. 

Analogously to the heavy-ball case, we establish convergence of the preconditioned Nesterov method under uniform boundedness of the preconditioner.

\begin{theorem}
    \label{thm:precNest}
    Consider the preconditioned Nesterov defined in
    \eqref{eq:PrecNesterov-def} for solving \eqref{eq:optprob-def}, and assume the preconditioning
    matrix $\Pk$ is uniformly bounded -- \ie~there exist $\Pmax>\Pmin>0$ such that $\Pmin I\preceq \Pk\preceq \Pmax I$ for all $k\in\mathbb{N}$.
    Then there exist a constant $\beta^*\in(0,1)$ and a positive
    function $\alpha^*:[0,\beta^*)\to\re_{>0}$ such that for every
    $\beta\in[0,\beta^*)$ and every $\alpha\in(0,\alpha^*(\beta))$,
    it holds that
    \begin{equation}
        \label{eq:PrecNestConvThm}
        \lim_{k\to\infty}\|\nabla \loss(\yk)\|=0.
    \end{equation}
    That is, the gradient evaluated at the extrapolated points converges to zero.
\end{theorem}

Under a global Polyak–\L{}ojasiewicz condition, this again implies linear convergence of the function values.


\begin{corollary}\label{cor:precNest-PL}
    Let $\loss$ satisfy a global Polyak--\L{}ojasiewicz inequality with constant $\mu>0$.
    Then, for every $\beta\in[0,\beta^\star)$ and every
    $\alpha\in(0,\alpha^\star(\beta))$ (where $\beta^*$ and
    $\alpha^*(\cdot)$ are as in Theorem~\ref{thm:precNest}), there exist constants
    $C>0$ and $\rho\in(0,1)$ such that
    \begin{equation}
        f(x_k)-f^\star \leq C\rho^k,
        \qquad \forall k\in\mathbb N.
    \end{equation}
\end{corollary}

Theorem \ref{thm:precNest} and Corollary \ref{cor:precNest-PL} extend the convergence guarantees of Section \ref{sec:Muon} to the Nesterov setting. In particular, they show that preconditioned Nesterov methods exhibit the same qualitative behavior as their heavy-ball counterparts, provided the preconditioner remains uniformly bounded.

We now specialize this framework to the preconditioning structure induced by Muon, leading to a Nesterov-based variant that we refer to as Muesterov.

\subsection{The Muesterov algorithm}

Replacing the heavy-ball momentum in Soft Muon with a Nesterov extrapolation step leads to the following iteration, which we refer to as \emph{Muesterov}.
\begin{subequations}
    \label{eq:Muesterov-def}
    \begin{align}
        \yk &= \xk + \beta (\xk - \xm), \\
        \xp &= \yk - \alpha (\Gk \Gk^\top + \eps^2 I)^{-1/2} \Gk,
    \end{align}
\end{subequations}
where $G_k := \nabla \loss(y_k)$. Note that this preconditioned Nesterov formulation uses an analogous preconditioner to Muon, since $(\Gk \Gk^\top + \eps^2 I)^{-1/2} \Gk\to\sign(\Gk)$ as $\eps\to 0$. Moreover, the update can be implemented using the same Newton-Schulz iteration as in Muon, allowing the method to inherit the same computational advantages and parallelization structure.


As in the heavy-ball case, the key step is to establish boundedness of the generated trajectory, which ensures uniform boundedness of the induced preconditioner. We provide this result in the following theorem.

\begin{theorem}
    \label{thm:Muesterov-Precompact}
    Let $\{\xk,\yk\}$ be the sequence generated by
    \eqref{eq:Muesterov-def} from the initial condition $\{x_0,x_0\}$.
    Assume $f$ is proper, $L$-smooth, has a unique global minimizer
    $x^*$, and satisfies a global P\L{} inequality with constant
    $\mu>0$.
    Then there exist $\beta^*\in(0,1]$ and a positive function
    $\alpha^*:[0,\beta^*)\to\re_{>0}$ such that for every
    $\beta\in[0,\beta^*)$ and every $\alpha\in(0,\alpha^*(\beta))$
    the sequences $\{\xk\}$ and $\{\yk\}$ remain bounded.
\end{theorem}

Once boundedness is established, the preconditioner remains uniformly bounded along the trajectory, allowing us to invoke Theorem~\ref{thm:precNest} to obtain convergence.

\begin{corollary}
    \label{cor:Muesterov-Convergence}
    For the algorithm presented in \eqref{eq:Muesterov-def}, there exists $\bar\beta\in(0,1)$ and $\bar\alpha:[0,\bar\beta)\to\re_+$, such that every $\beta\in[0,\bar\beta)$ and every
    $\alpha\in(0,\bar\alpha(\beta))$ the sequences $\{\xk\}$ and
    $\{\yk\}$ are bounded and $\lim_{k\to\infty}\|\dLoss[\yk]\|=0$, i.e. the gradient evaluated at the extrapolated points converges to zero.
\end{corollary}

We have thus extended the convergence framework developed for Muon to a Nesterov-based setting, showing that normalization-based preconditioning is compatible with accelerated momentum schemes. Crucially, Muesterov preserves the normalization-based structure of Muon, differing only in the point at which the gradient is evaluated. This provides both a theoretical justification for the proposed algorithm and a principled foundation for comparing Muon and Muesterov.

\section{Simulations}

\subsection{Building intuition through a synthetic example}

\begin{figure}
    \centering
    \includegraphics[width=0.2\linewidth]{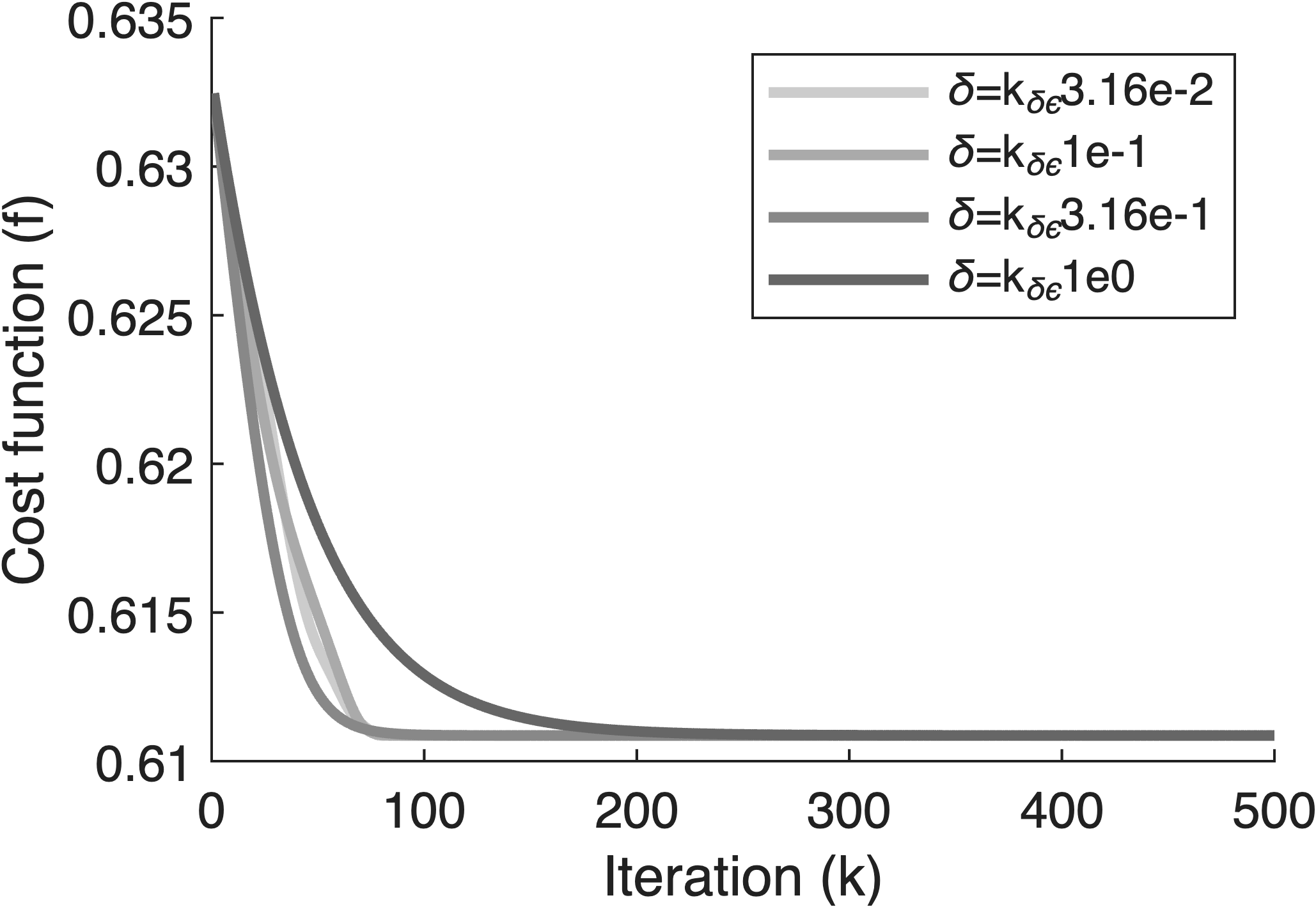}
    \includegraphics[width=0.2\linewidth]{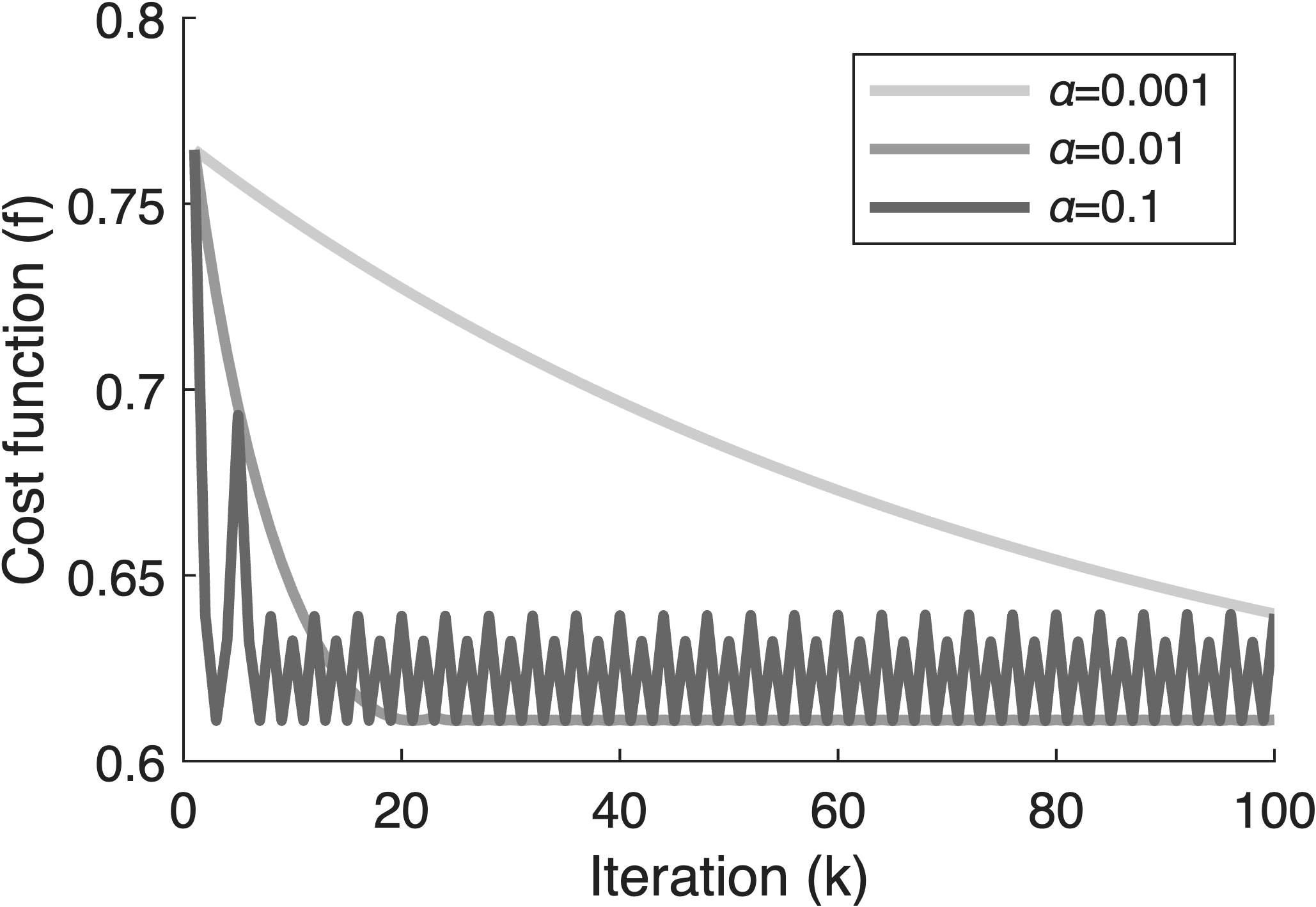}
    \includegraphics[width=0.2\linewidth]{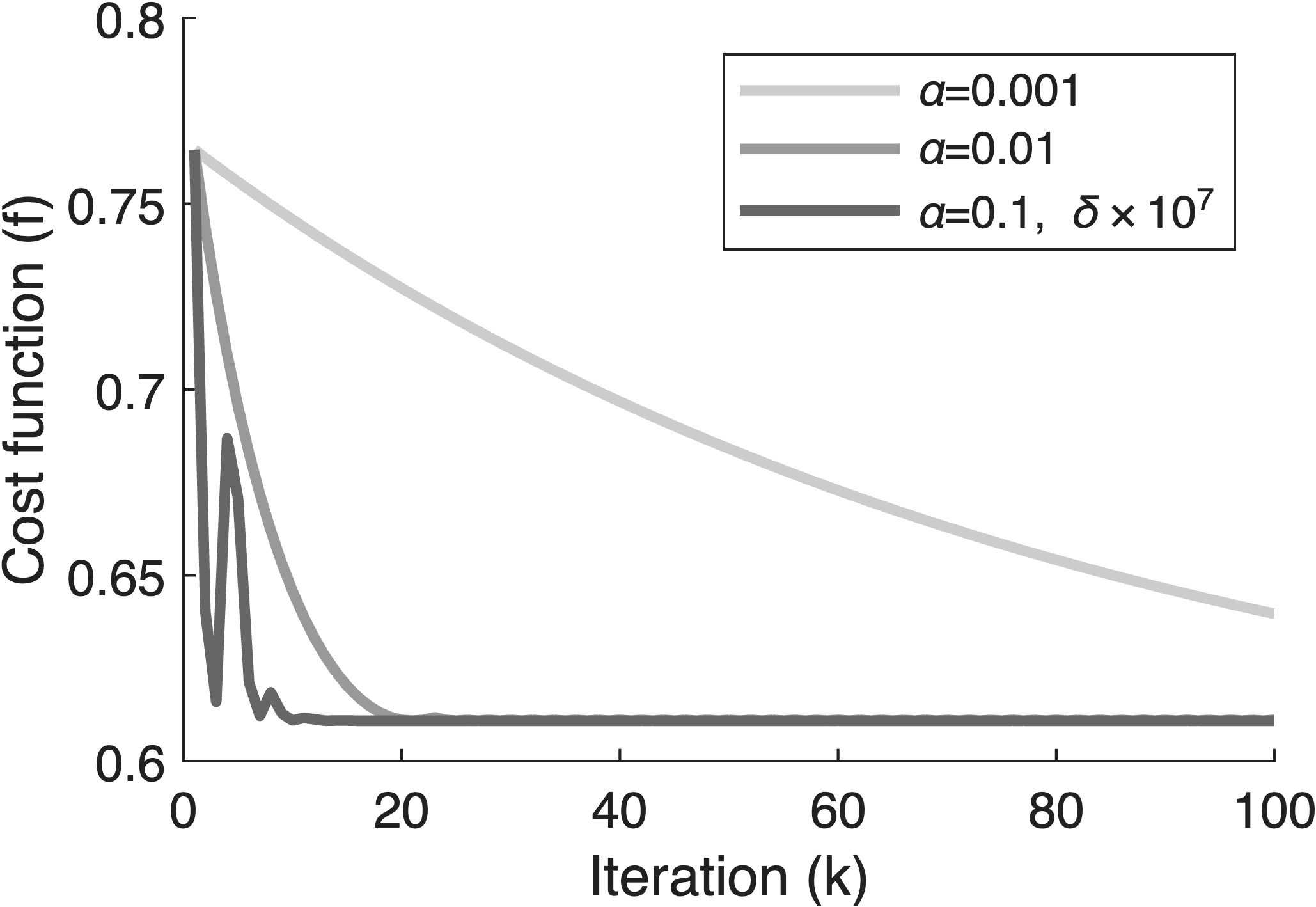}
    \includegraphics[width=0.2\linewidth]{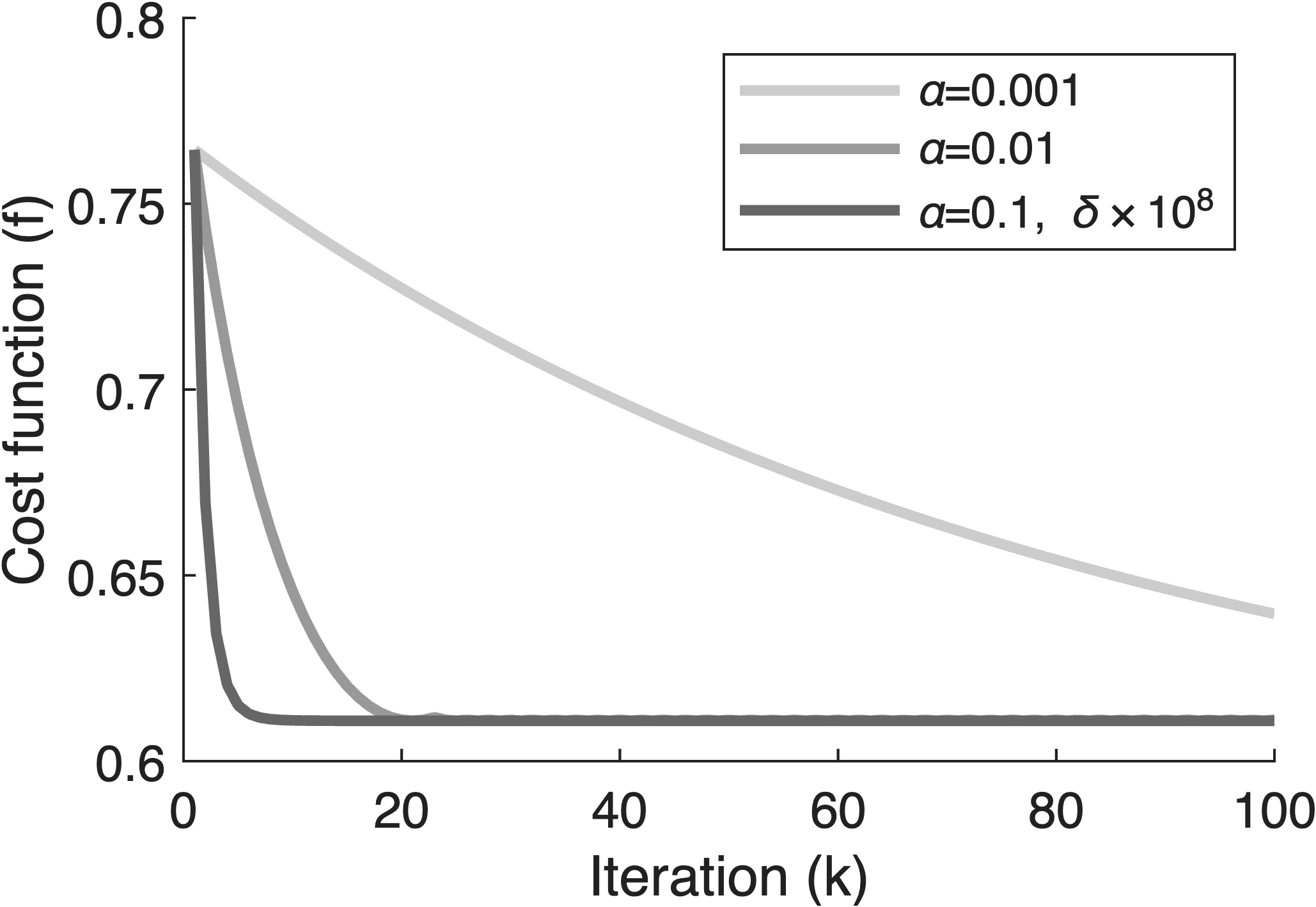}\\
    \includegraphics[width=0.2\linewidth]{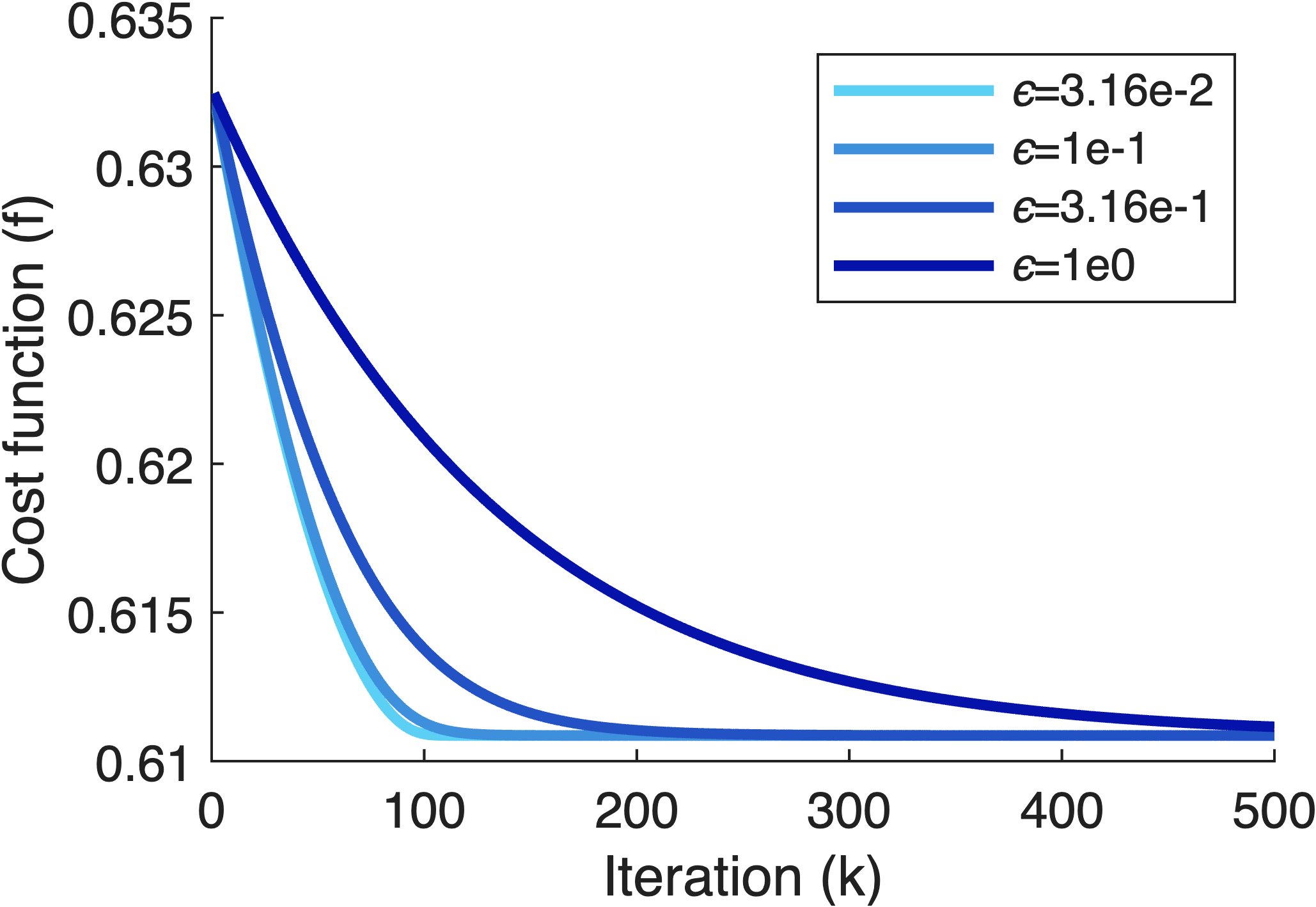}
    \includegraphics[width=0.2\linewidth]{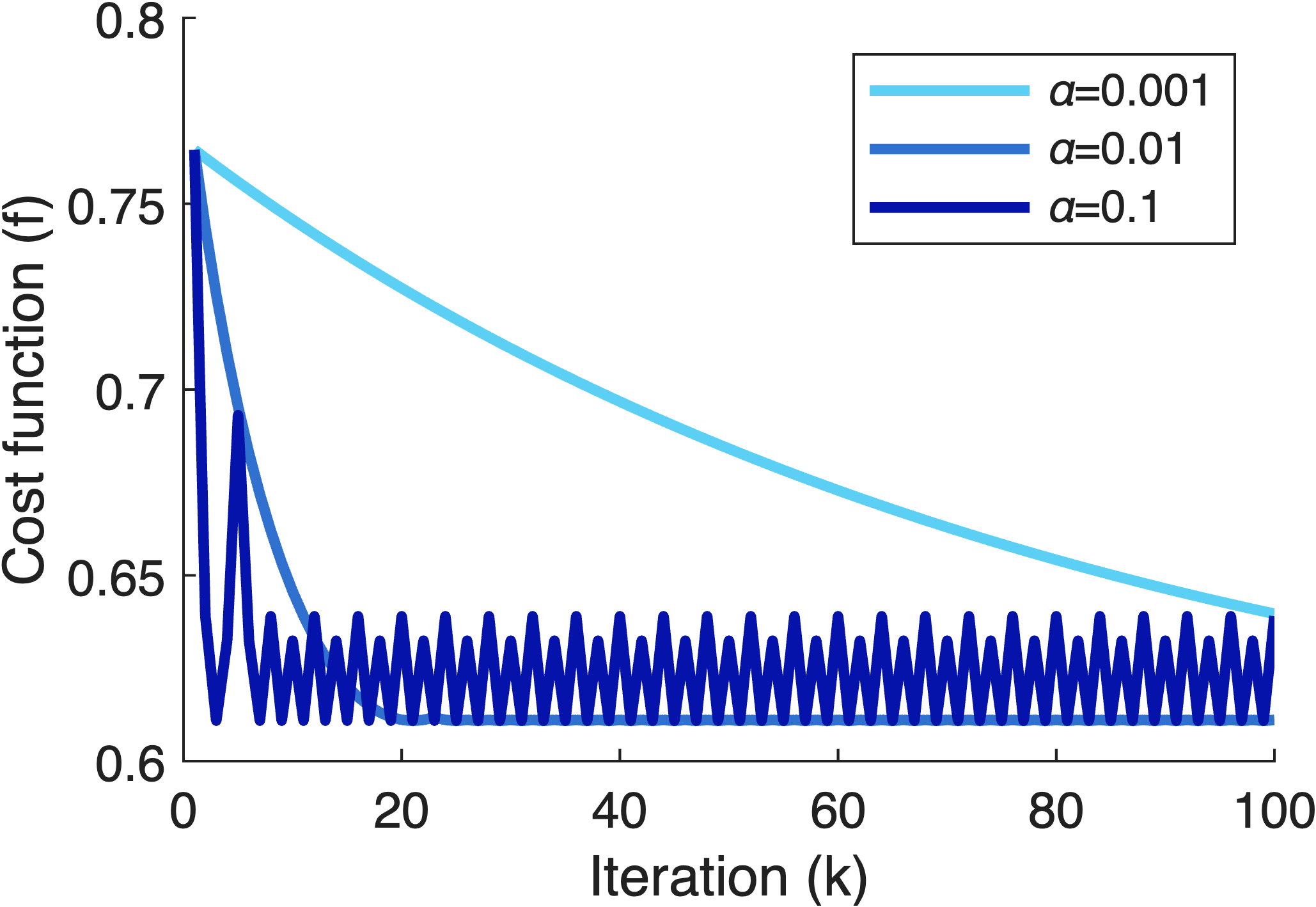}
    \includegraphics[width=0.2\linewidth]{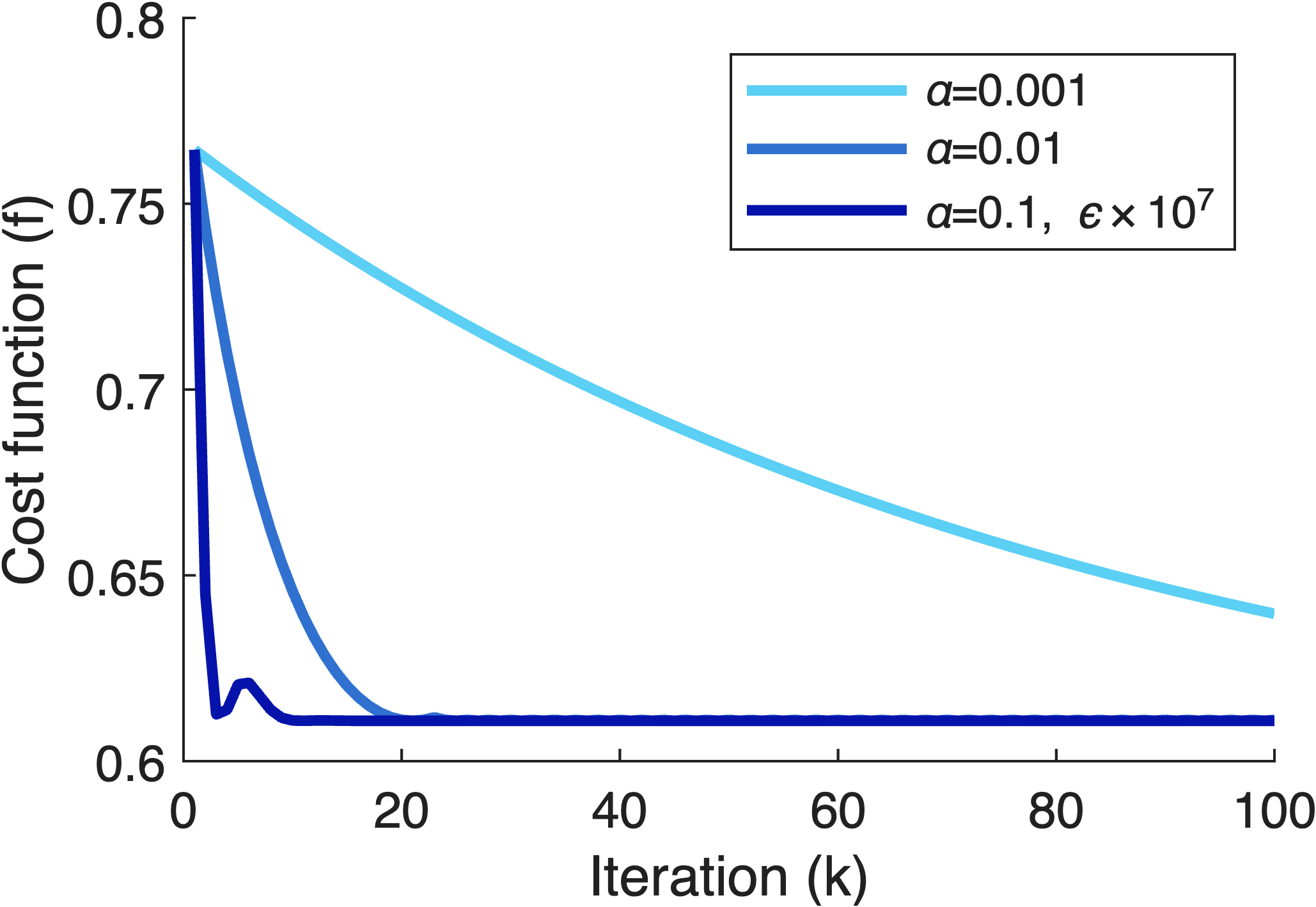}
    \includegraphics[width=0.2\linewidth]{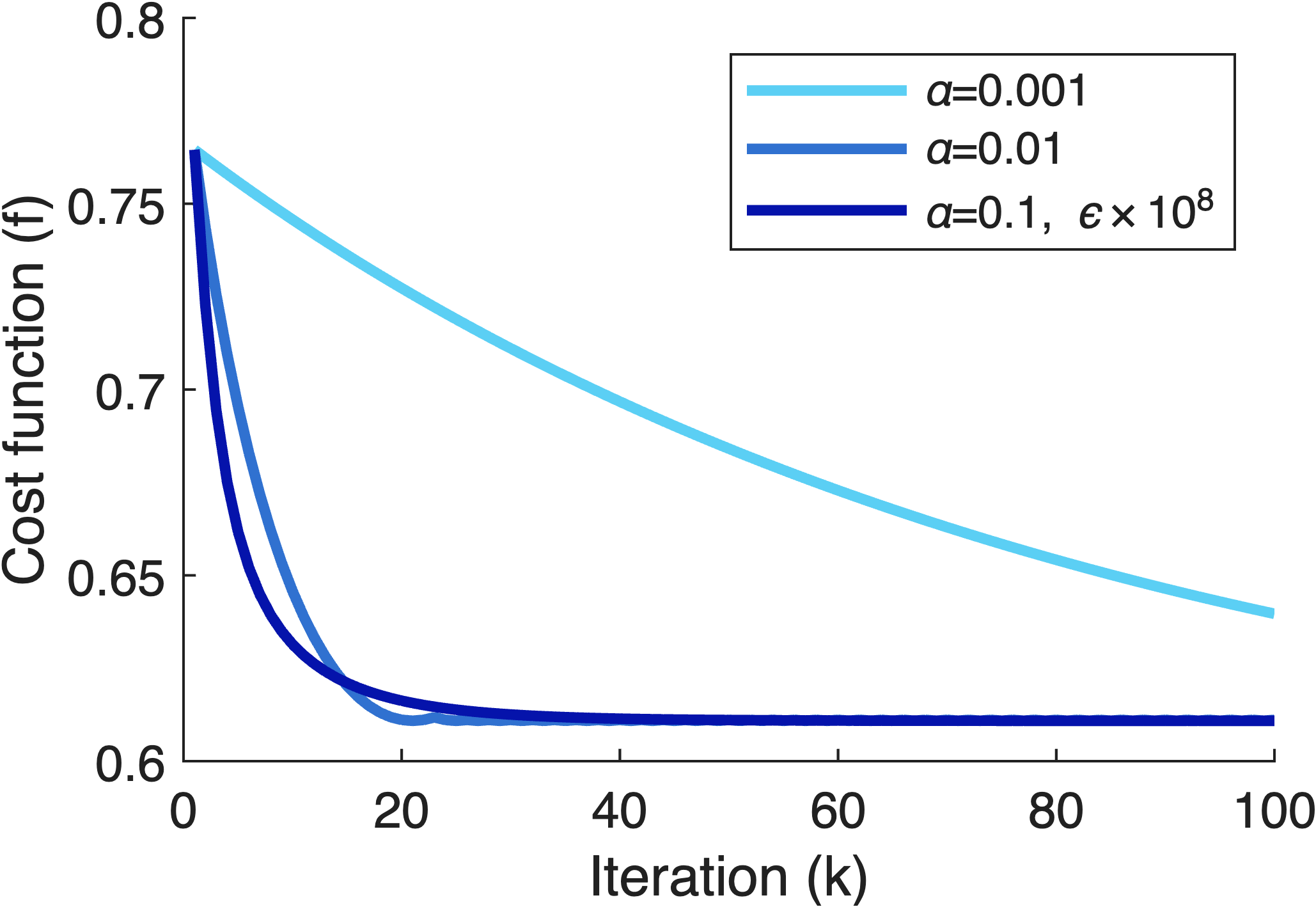}\\
    \includegraphics[width=0.2\linewidth]{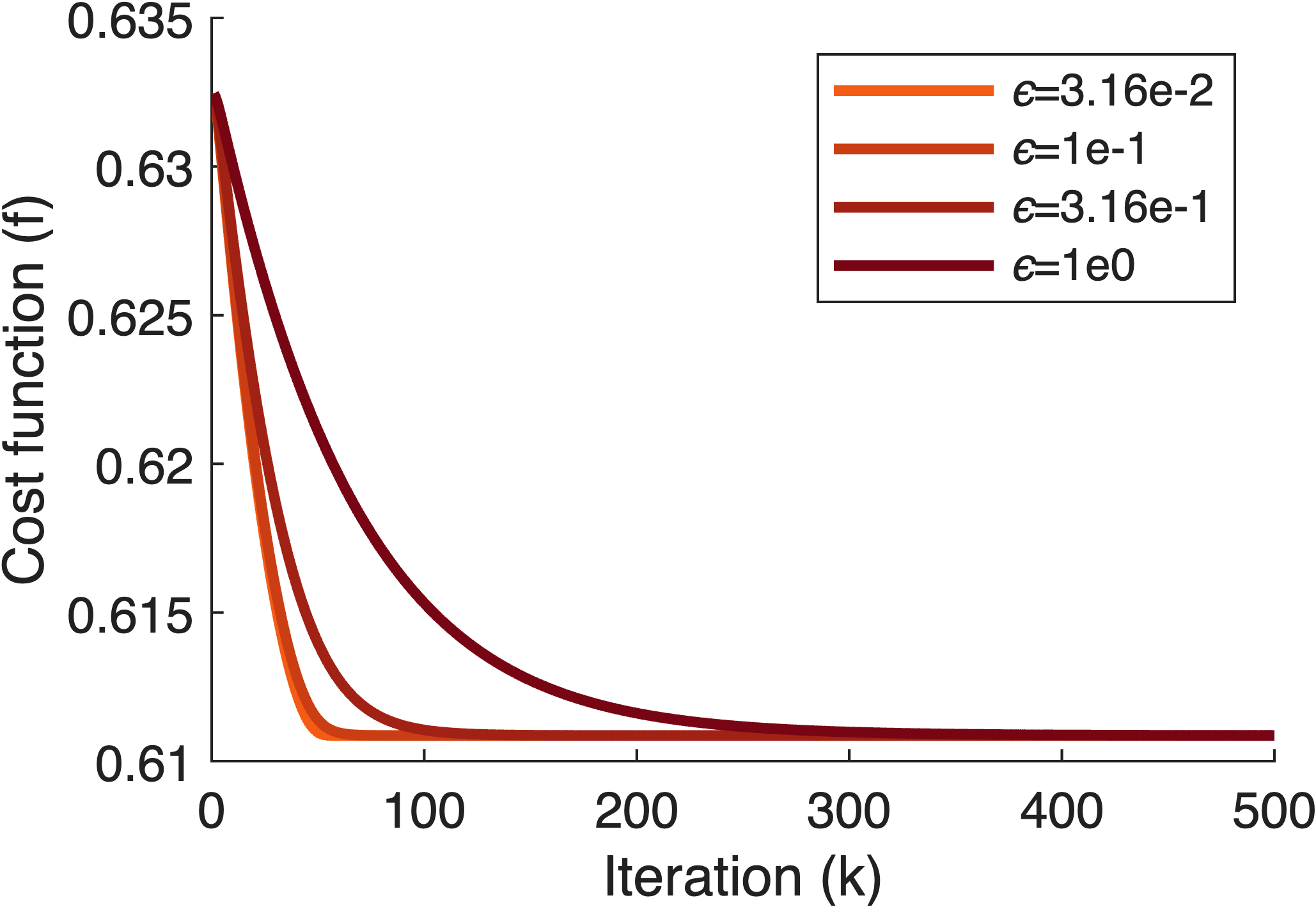}
    \includegraphics[width=0.2\linewidth]{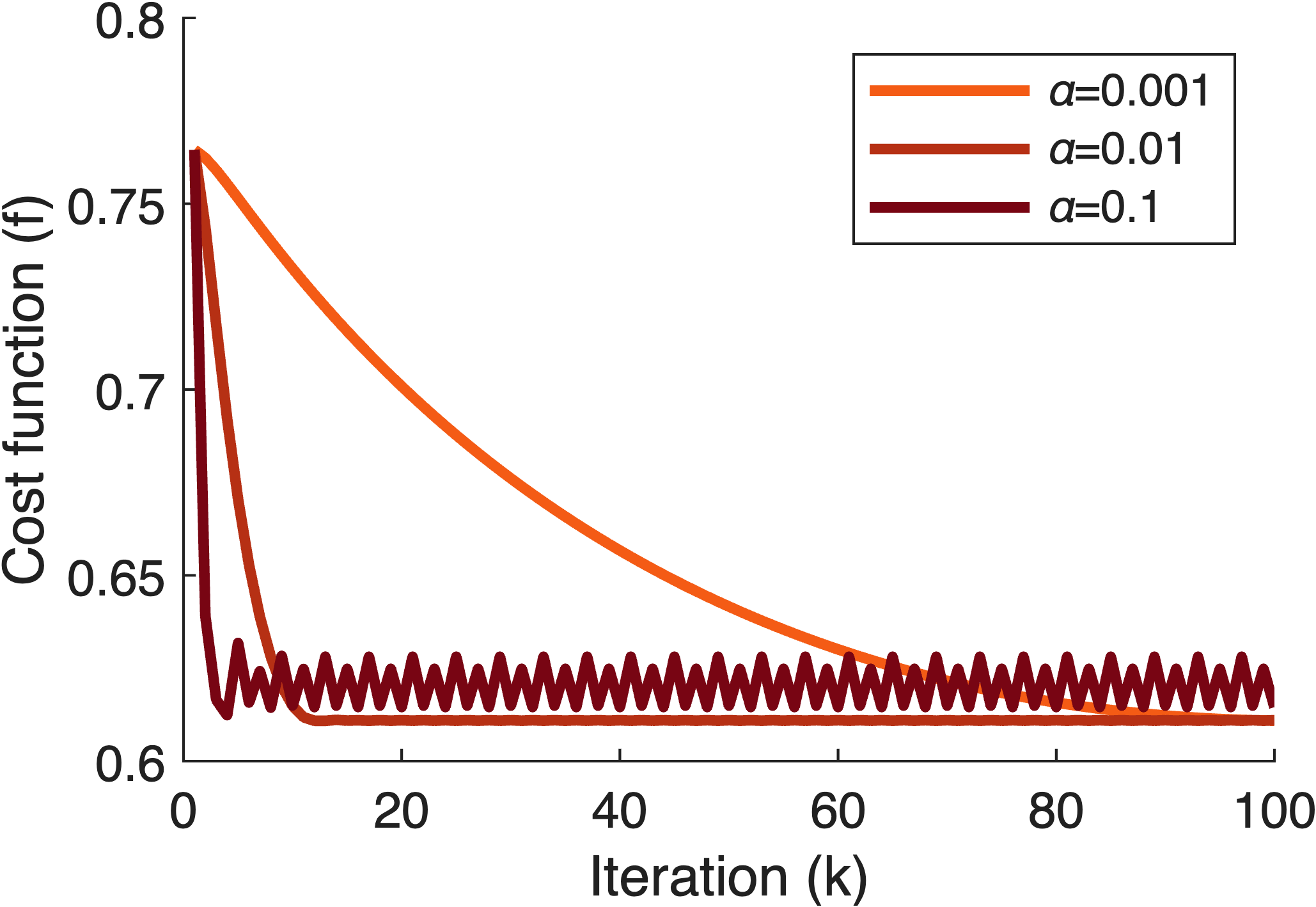}
    \includegraphics[width=0.2\linewidth]{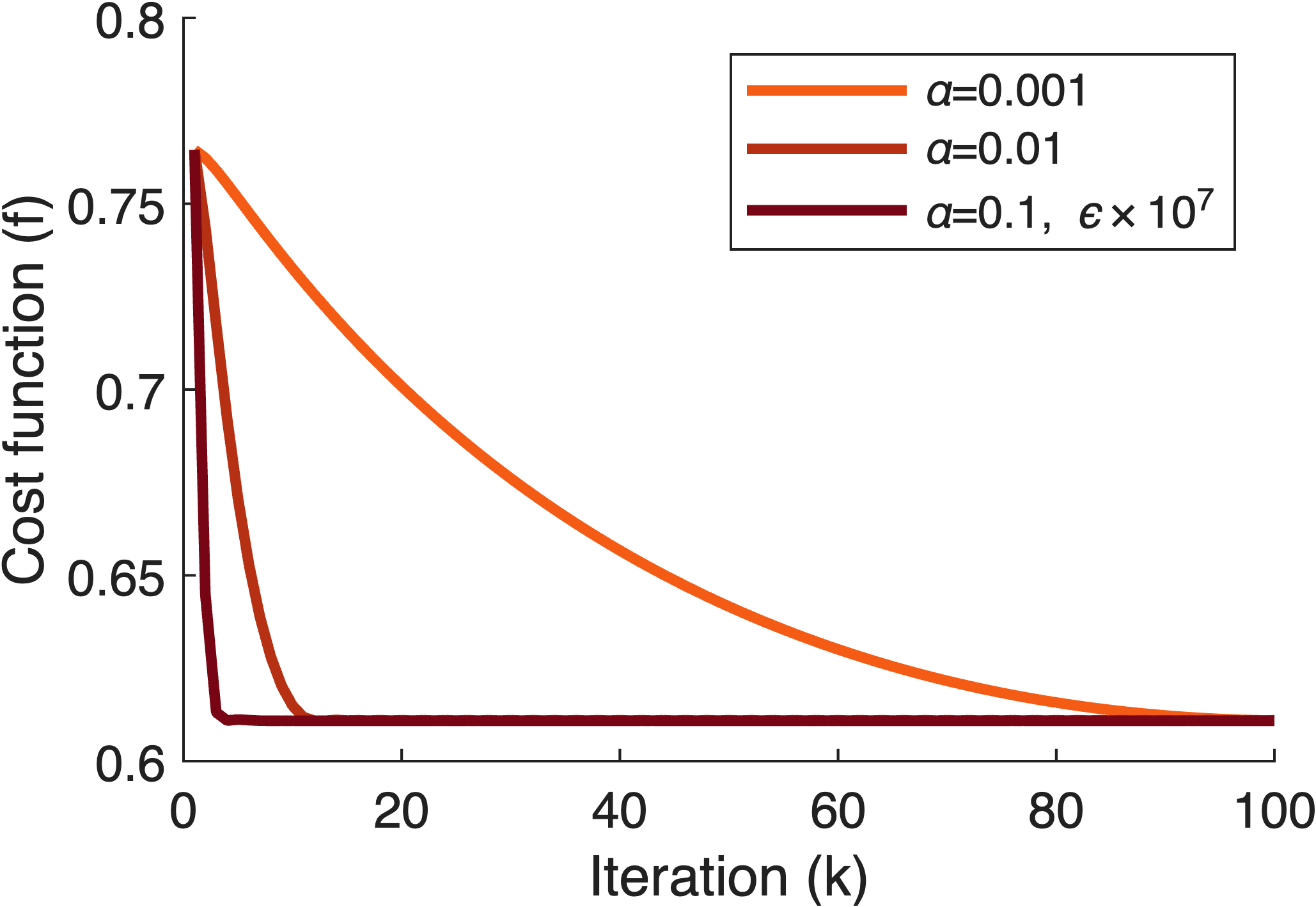}
    \includegraphics[width=0.2\linewidth]{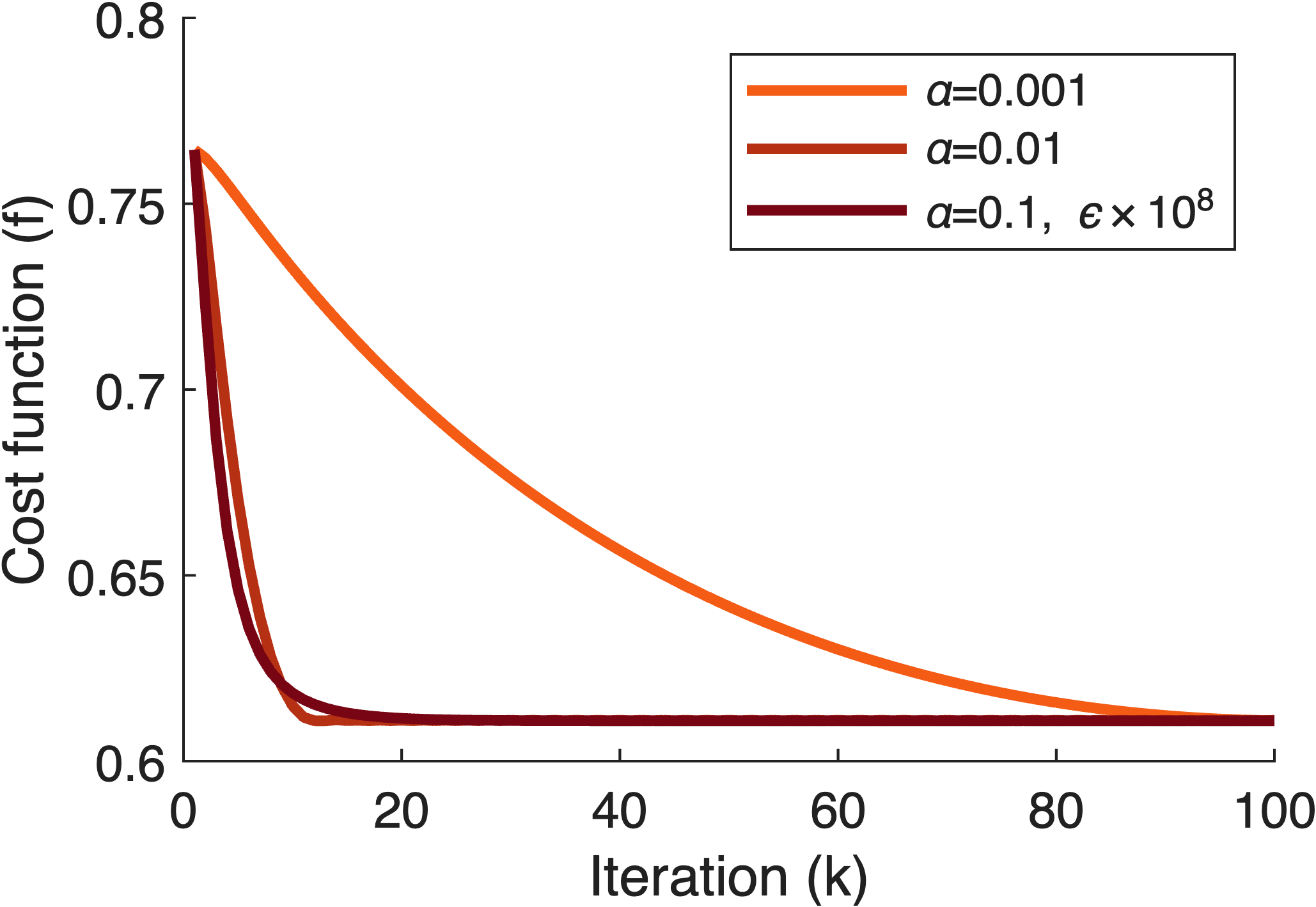}\\
    ~~~~(a)\qquad\qquad\qquad~~~~(b)\quad\qquad\qquad\qquad(c)\qquad\qquad\qquad~~~(d)
    \caption{Synthetic simulations for Canonical Muon \eqref{eq:Muon-Orig} (first row, gray), Soft Muon \eqref{eq:Muon-Soft} (second row, blue), and Muesterov \eqref{eq:Muesterov-def} (third row, red). Column (a) depicts four training cost trajectories for $\alpha=10^{-3}$, $\beta=0.5$, and increasing values of $\eps$. Column (b) depicts the training cost for $\beta=0.5$, $\eps=10^{-7}$, and increasing values of $\alpha$. Columns (c) and (d) repeat the same simulations as column (b) but for much larger values of $\eps$ for the larger $\alpha$.}
    \label{fig:SyntheticSimulationSet}
\end{figure}

To build intuition for the role of the hyperparameters $\alpha$ and $\eps$, we consider a simple scalar cross-entropy minimization problem for which closed-form expressions are available. This setting isolates the effect of normalization and regularization, and provides a transparent testbed for understanding the dynamics of Muon, Soft Muon, and Muesterov.

Specifically, let the true label follow a Bernoulli distribution $B_p$ with $p=0.7$, and consider the parametric family $q_{\theta}(X=1)=\theta$, $q_{\theta}(X=0)=1-\theta$. The associated cross-entropy loss and its derivative are given by $\loss(\theta)=-p\log(\theta)-(1-p)\log(1-\theta)$, and $\loss'(\theta)=-\frac{p}{\theta}+\frac{1-p}{1-\theta}$,
%
%
with unique minimizer $\theta^*=p$. Furthermore, $\loss$ satisfies a global P\L{} inequality on $(0,1)$ with constant $\mu = 2\left(p^{1/3}+(1-p)^{1/3}\right)^3$. In this scalar setting, the soft-sign operator makes the effect of $\eps$ particularly transparent, $\SSat(h,\eps)={h}/{\eps}+O(h^3)~ \text{near } h=0$, while $\SSat(h,\eps)\approx \sign(h)
~ \text{when } |h|\gg \eps$.

Figure~\ref{fig:SyntheticSimulationSet} compares Canonical Muon, Soft Muon, and Muesterov under different choices of $\alpha$ and $\eps$. For small regularization, all methods behave similarly and exhibit persistent oscillations characteristic of nearly sign-based updates. Increasing $\eps$ transitions Soft Muon and Muesterov toward a gradient-like regime with effective gain proportional to $1/\eps$, stabilizing the dynamics and restoring convergence. Canonical Muon exhibits the same qualitative transition: for small $\delta$, its trajectories are nearly indistinguishable from Soft Muon, while larger $\delta$ introduces quantitative differences due to the higher-order structure of the Newton-Schulz polynomial.

The experiments further suggest that $\alpha$ and the normalization regularizer must be tuned jointly. Large learning rates combined with weak regularization produce sustained oscillations, whereas increasing $\eps$ (or $\delta$) restores stable behavior and can substantially accelerate convergence relative to the nearly sign-based regime. Excessively large regularization, however, eventually slows convergence. Experiments also indicate potential advantage of using Muesterov over Muon, with the former reaching lower loss values faster than both Canonical and Soft Muon under identical hyperparameters in a simple 1D loss simulation.

We next examine the problem of training GPT models, one of the primary practical applications of Muon.

\subsection{Training nanoGPT}
\label{sec:nanogpt-simulations}

\begin{figure}[t]
    \centering
    \includegraphics[width=0.8\linewidth]{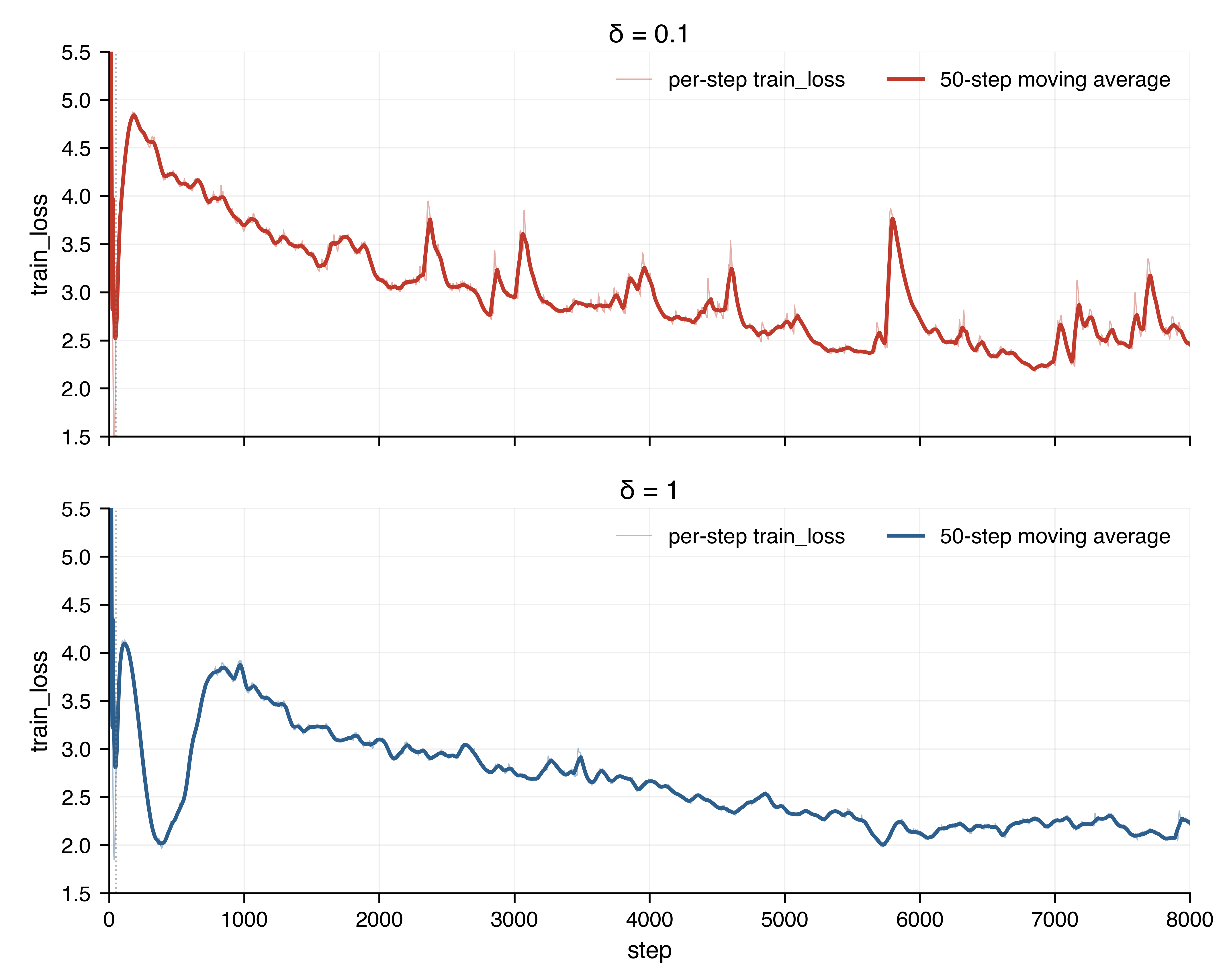}
    \caption{                                                                        
  Fixed-batch nanoGPT overfitting at $100\times$ the baseline script-level         
  learning rate ($\alpha = 0.36$ instead of\ $0.0036$). The Muon optimizer trains the     
  transformer blocks at $0.1\alpha$ while AdamW remains at the baseline rate       
  ($0.0036$) on the output head. After the first minibatch is drawn it is          
  reused at every step, so the plotted train loss is a fixed-batch diagnostic      
  rather than a sampled training curve. Faint traces show the per-step values;     
  bold traces show a $50$-step moving average.                    
  \textbf{Top Panel.} With Newton-Schulz regularizer $\delta = 0.1$, the trajectory      
  exhibits sustained peak--trough oscillations superimposed on an overall
  descent.                                                                         
  \textbf{Bottom Panel.} With $\delta = 1$, the same learning rate produces a single
  early overshoot followed by smooth, near-monotone descent toward a fixed         
  point. The contrast is consistent with the soft-sign / Newton-Schulz            
  analysis: larger $\delta$ reduces the effective gain of the normalized           
  update and damps the oscillatory regime.                                         
  }
    \label{fig:nanogpt-fixed-batch}
\end{figure}

The preceding experiments isolate, in a low-dimensional setting, how the
learning rate $\alpha$ and the regularization in the soft-sign/Newton-Schulz
normalization interact to produce either oscillatory behavior or convergence.
We next investigate whether the same qualitative mechanism is visible in a practical transformer-training setting. Rather than performing a standard large-scale training benchmark, the goal of this experiment is to evaluate if intuition from theory is verifiable using a real model: we intentionally construct a controlled setting in which oscillatory behavior becomes visible, allowing us to directly examine the role of the Newton-Schulz regularizer in the practical Muon implementation. 

To achieve this goal, we edited the original script from \citep{modded_nanogpt_2025} in the following ways: we increased the script-level learning rate from $0.0036$ to $0.36$, with the AdamW learning rate for the output-head parameters kept
at its baseline value, of $0.0036$; we also used a short $50$-iteration learning rate warmup, removed the warmdown, and fixed all seeds. Crucially, we also implemented the optimizer to be fixed to only one batch: after the initial batch is drawn, the same batch is reused throughout training. The plotted training loss should therefore be read as a fixed-batch loss, rather than as a standard minibatch-sampled training curve.

Figure~\ref{fig:nanogpt-fixed-batch} shows that the smaller-regularization run develops visible oscillatory behavior. Although the loss decreases overall, the trajectory is repeatedly interrupted by sharp spikes followed by recovery. Since the same batch is reused throughout the run, these oscillations are not caused
by changing minibatches. Rather, they are consistent with the regularization mechanism identified in the synthetic experiments and convergence analysis.

Increasing the regularizer from $\delta=0.1$ to $\delta=1$ substantially smooths
the trajectory. The run still has an initial transient, but the repeated
large-amplitude peak-trough events are strongly reduced. This agrees with the
soft-sign interpretation of the Newton-Schulz regularizer. A larger
regularizer reduces the effective gain of the normalized update when the
momentum matrix is small relative to the normalization scale. Thus, while this fixed-batch nanoGPT experiment is intentionally diagnostic rather than benchmark-oriented, it nevertheless exhibits the same qualitative $\alpha$-regularization interaction predicted by the synthetic experiments and theoretical analysis.

\section{Conclusion}

In this paper we established the first asymptotic convergence guarantees for Muon by interpreting the regularization implicit in its Newton-Schulz implementation as a bounded preconditioning mechanism. This perspective exposes Muon as a preconditioned Polyak heavy-ball method and enables a classical Lyapunov analysis under suitable hyperparameter choices. Motivated by the same interpretation, we introduced Muesterov, a Nesterov-based variant that preserves the normalization structure of Muon while enjoying analogous convergence guarantees.

Beyond the theoretical results, the synthetic and transformer-training experiments suggest that the interaction between the learning rate and the Newton-Schulz regularizer plays a central role in controlling the transition between oscillatory sign-like dynamics and stable convergence. Understanding how these regularization mechanisms behave in practical large-scale training regimes remains an important direction for future work.

\bibliographystyle{plainnat}
\bibliography{references}

\appendix
\section{Proofs of theoretical claims}

\subsection{Proof of Proposition \ref{prop:Muon-OrigIdeal-Equivalence}}

\paragraph{Statement.} Let $\NS(\Bk)=\sign(\Bk)$. Also, for any $x_0\in\PSp$ let $\{\xktilde,\Bktilde\}$ and $\{\xkbar,\Hkbar\}$ be the sequences generated by \eqref{eq:Muon-Orig} and \eqref{eq:Muon-Ideal} respectively, each initialized at $\{x_0, B_0\}$ and $\{x_0, \Hz\}$ with $B_0=(1-\beta)^{-1}\Hz$. Then, for every $k\in\mathbb{N}$ it holds that $\xktilde=\xkbar$ and $\tilde \Bk=(1-\beta)^{-1}\Hkbar$.
\begin{proof}
    The proof is simply algebraic. First note that for any given sequence $\{\xk\}\subset\PSp$ and for $B_0=(1-\beta)^{-1}\Hz$ it holds that
    \begin{equation*}
        \Hk=\beta^k \Hz+(1-\beta)\sum_{i=0}^{k-1}\beta^{k-1-i}\dLoss[\x_i]
        =(1-\beta)\left(\beta^k B_0+\sum_{i=0}^{k-1}\beta^{k-1-i}\dLoss[\x_i]\right)
        =(1-\beta)\Bk,
    \end{equation*}
    and thus $\sign(\Hk)=\sign((1-\beta)\Bk)=\sign(\Bk)$. Since both $\xktilde$ and $\xkbar$ are initialized at the same point and have the same update law, it implies that $\{\xkbar,\Hkbar\}=\{\xktilde,(1-\beta)\Bktilde\}$, proving the claim.
\end{proof}

\subsection{Proof of Proposition \ref{prop:Muon-OrigSoft-Equiv}}

\paragraph{Statement.} Let $\NS(\Bk)=\SSat(\Bk,\eps)$ for some $\eps>0$. Also, for any $x_0\in\PSp$ let $\{\xktilde,\Bktilde\}$ and $\{\xkbar,\Hkbar\}$ be the sequences generated by \eqref{eq:Muon-Orig} and \eqref{eq:Muon-Soft} respectively, with each initialized at $\{x_0, B_0\}$ and $\{x_0, \Hz\}$ with $B_0=(1-\beta)^{-1}\Hz$ and with constant $\eps(1-\beta)$ in \eqref{eq:Muon-Soft}. Then, for every $k\in\mathbb{N}$ it holds that $\xktilde=\xkbar$ and $\tilde \Bk=(1-\beta)^{-1}\Hkbar$.
\begin{proof}
    The proof is, again, simply algebraic and follows almost identically to the previous proposition. First note that for any given sequence $\{\xk\}\subset\PSp$ and for $B_0=(1-\beta)^{-1}\Hz$ it holds that
    \begin{equation*}
        \Hk=\beta^k \Hz+(1-\beta)\sum_{i=0}^{k-1}\beta^{k-1-i}\dLoss[\x_i]
        =(1-\beta)\left(\beta^k B_0+\sum_{i=0}^{k-1}\beta^{k-1-i}\dLoss[\x_i]\right)
        =(1-\beta)\Bk,
    \end{equation*}
    and thus $\SSat(\Hk,\eps(1-\beta))=\SSat((1-\beta)\Bk,\eps(1-\beta))=\SSat(\Bk,\eps)$. Since both $\xktilde$ and $\xkbar$ are initialized at the same point and have the same update law, it implies that $\{\xkbar,\Hkbar\}=\{\xktilde,(1-\beta)\Bktilde\}$, proving the claim.
\end{proof}

\subsection{Proof of Theorem \ref{thm:PrecHeavyBall}}

\paragraph{Statement.} Consider the preconditioned heavy-ball iteration presented in~\eqref{eq:PrecHeavyBall}, and assume
the preconditioning matrix $\Pk$ is uniformly bounded  -- \ie~there exist $\Pmax>\Pmin>0$ such that $\Pmin I\preceq \Pk\preceq \Pmax I$ for all $k\in\mathbb{N}$.
Then there exist a constant $\beta^\star\in(0,1)$ and a positive function
$\alpha^\star:[0,\beta^\star)\to\R_{>0}$ such that for every $\beta\in[0,\beta^\star)$
and every $\alpha\in(0,\alpha^\star(\beta))$, it holds that
\begin{equation*}
  \lim_{k\to\infty}\norm{\dLoss[\xk]} = 0.
\end{equation*}

\begin{proof}
Let $f^\star:=\inf_{x}f(x)$ and define $e_k := \Hk - \dLoss[\xk]$,
which satisfies $e_0 = 0$ by initialization. Henceforth adopt the shorthands
$\Phi_k := \Phi(\xk,e_k)$, $G_k := \dLoss[\xk]$, and $d_k := \Pk^{-1}\Hk$,
and define the Lyapunov function
\begin{equation}
  \Phi(\xk, e_k) := \loss(\xk) - f^\star + A\norm{e_k}^2, \label{eq:lyapunov}
\end{equation}
with $A := 1/(4L)$. We prove convergence by showing that for sufficiently small $\alpha$ and $\beta$, $\Phi_k$ is a Lyapunov
function of the discrete-time system~\eqref{eq:PrecHeavyBall}, i.e.\ that
$\Phi_{k+1} - \Phi_k < 0$.

We first look at $\loss(\xp) - \loss(\xk)$. By $L$-smoothness of $f$,
\[
  \loss(\xp) \leq \loss(\xk)
    - \alpha\inner{G_k}{d_k}
    + \frac{L\alpha^2}{2}\norm{d_k}^2.
\]
Writing $\Hk = G_k + e_k$ and using $(1/\Pmax)\Id\preceq \Pk^{-1}\preceq(1/\Pmin)\Id$,
\[
  \inner{G_k}{d_k}
    = \inner{G_k}{\Pk^{-1}G_k} + \inner{G_k}{\Pk^{-1}e_k}
    \geq \frac{1}{\Pmax}\norm{G_k}^2 - \frac{1}{\Pmin}\norm{G_k}\norm{e_k},
\]
and $\norm{d_k} \leq (1/\Pmin)\norm{\Hk} \leq (1/\Pmin)(\norm{G_k}+\norm{e_k})$, so
\begin{equation}
  \loss(\xp) - \loss(\xk)
    \leq -\!\left(\frac{\alpha}{\Pmax} - \frac{L\alpha^2}{\Pmin^2}\right)\norm{G_k}^2
         + \frac{\alpha}{\Pmin}\norm{G_k}\norm{e_k}
         + \frac{L\alpha^2}{\Pmin^2}\norm{e_k}^2. \label{eq:fdescent}
\end{equation}

Next, consider $\norm{e_{k+1}}^2 - \norm{e_k}^2$. From the preconditioned heavy-ball
recursion~\eqref{eq:PrecHeavyBall} and the definition of $e_k$,
\[
  e_{k+1} = \beta e_k + \bigl(G_k - G_{k+1}\bigr),
\]
and by $L$-smoothness, $\norm{G_{k+1} - G_k} \leq L\alpha\norm{d_k} \leq
(L\alpha/\Pmin)(\norm{G_k}+\norm{e_k})$. Applying $(a+b)^2\leq 2a^2 + 2b^2$,
\begin{equation}
  A\bigl(\norm{e_{k+1}}^2 - \norm{e_k}^2\bigr)
    \leq \frac{4AL^2\alpha^2}{\Pmin^2}\norm{G_k}^2
    - A\!\left(1 - 2\beta^2 - \frac{4L^2\alpha^2}{\Pmin^2}\right)\norm{e_k}^2.
  \label{eq:edelta}
\end{equation}

Writing the iteration for $\Phi_k$ by adding~\eqref{eq:fdescent}
and~\eqref{eq:edelta} results in
\begin{equation}
  \Phi_{k+1} - \Phi_k
    \leq -\underbrace{\left(\frac{\alpha}{\Pmax} - \frac{2L\alpha^2}{\Pmin^2}\right)}_{=:\,C_G}
          \norm{G_k}^2
         +\underbrace{\frac{\alpha}{\Pmin}}_{=:\,C_c}\norm{G_k}\norm{e_k}
         -\underbrace{\left(\frac{1-2\beta^2}{4L} - \frac{2L\alpha^2}{\Pmin^2}\right)}_{=:\,C_e}
          \norm{e_k}^2,
  \label{eq:lyapdiff}
\end{equation}
where we used $A = 1/(4L)$ to collect the $2L\alpha^2/\Pmin^2$ factor appearing
identically in both the $\norm{G_k}^2$ and $\norm{e_k}^2$ terms.
Now $\Phi_{k+1} - \Phi_k \leq 0$ whenever the quadratic form
$C_G\norm{G_k}^2 + C_e\norm{e_k}^2 - C_c\norm{G_k}\norm{e_k}$ is positive
semi-definite, which by Sylvester's criterion holds if and only if $C_G > 0$,
$C_e > 0$, and $4C_GC_e > C_c^2$.

For fixed $\beta\in[0,1)$, the condition $C_G > 0$ reads
\[
  \alpha < \frac{\Pmin^2}{2L\Pmax} =: \alpha_1,
\]
and the condition $C_e > 0$ reads
\[
  \alpha < \frac{\Pmin}{2L}\sqrt{\frac{1-2\beta^2}{2}} =: \alpha_2(\beta),
\]
with no constraint for $\beta = 0$ and $\alpha_2(\beta) > 0$ for all
$\beta < 1/\sqrt{2}$. Since $4C_GC_e - C_c^2$ vanishes at $\alpha = 0$,
we factor out $\alpha > 0$ and the discriminant condition $4C_GC_e > C_c^2$ is
equivalent to $Q(\alpha) > 0$, where
\begin{equation}
  Q(\alpha) := \frac{16L^2}{\Pmin^4}\,\alpha^3
             - \frac{8L}{\Pmax\Pmin^2}\,\alpha^2
             + \frac{4\beta^2-3}{\Pmin^2}\,\alpha
             + \frac{1-2\beta^2}{L\Pmax}. \label{eq:Q}
\end{equation}
The constant term in~\eqref{eq:Q} is strictly positive if and only if
$\beta^2 < 1/2$, i.e.\ $\beta < \beta^\star$ where
\begin{equation}
  \beta^\star := \frac{1}{\sqrt{2}} \in (0,1). \label{eq:betastar}
\end{equation}
Since the leading coefficient of~\eqref{eq:Q} is positive and the constant term
is positive for $\beta < \beta^\star$, we have $Q(0)>0$ and $Q(\alpha)\to-\infty$ as $\alpha\to-\infty$, so $Q$ has at least one negative root. Moreover, $Q(\alpha_1)<0$ and $Q(\alpha)\to+\infty$ as $\alpha\to+\infty$, hence $Q$ has two positive roots $\alpha_-(\beta) \leq \alpha_+(\beta)$ and is positive on $(0,\alpha_-(\beta))$.

To verify that $\alpha_-(\beta) < \alpha_1$, note that at $\alpha=\alpha_1$ we have $C_G=0$, so $4C_GC_e - C_c^2 = -C_c^2 < 0$, which implies $Q(\alpha_1)<0$. Hence $\alpha_1 \in [\alpha_-(\beta),\alpha_+(\beta)]$ and $\alpha_-(\beta)<\alpha_1$. By the same argument, $Q(\alpha_2(\beta))<0$, so $\alpha_-(\beta)<\alpha_2(\beta)$.

Setting $\alpha^\star(\beta):=\alpha_-(\beta)$, all three conditions $C_G>0$, $C_e>0$, and $4C_GC_e>C_c^2$ hold for every $\alpha\in(0,\alpha^\star(\beta))$.

It remains to convert positive definiteness of the quadratic form into a bound with
independent squared terms. Applying Young's inequality with
$\delta = \sqrt{C_e/C_G} > 0$ to the cross term gives
\[
  C_c\norm{G_k}\norm{e_k}
    \leq \frac{C_c}{2}\sqrt{\frac{C_G}{C_e}}\norm{G_k}^2
       + \frac{C_c}{2}\sqrt{\frac{C_e}{C_G}}\norm{e_k}^2,
\]
and therefore
\begin{equation}
  \Phi_{k+1} - \Phi_k \leq -c_1\norm{G_k}^2 - c_2\norm{e_k}^2, \label{eq:descent}
\end{equation}
where $c_1 := C_G - (C_c/2)\sqrt{C_G/C_e} > 0$ and
$c_2 := C_e - (C_c/2)\sqrt{C_e/C_G} > 0$, both of which are strictly positive
since $4C_GC_e > C_c^2$ implies $2\sqrt{C_GC_e} > C_c$. Since $\Phi_k \geq 0$
by definition, summing over $k$ gives
\[
  \sum_{k=0}^{\infty}\norm{G_k}^2 \leq \frac{\Phi_0}{c_1} < \infty,
\]
and in particular $\norm{G_k}^2 \to 0$, which proves the statement.
\end{proof}

\subsection{Proof of Corollary \ref{cor:MuonSoft-LinearConvergence}}

\paragraph{Statement.} Let $f$ satisfy a global Polyak--\L{}ojasiewicz inequality with constant $\mu>0$, and let $P_k$ be uniformly bounded.
    Then, for every $\beta\in[0,\beta^\star)$ and every
    $\alpha\in(0,\alpha^\star(\beta))$ (where $\beta^\star$ and $\alpha^\star(\cdot)$ are as in Theorem \ref{thm:PrecHeavyBall}), there exist constants
    $C>0$ and $\rho\in(0,1)$ such that
    \begin{equation}
        f(x_k)-f^\star \leq C\rho^k,
        \qquad \forall k\in\mathbb N.
    \end{equation}

\begin{proof}
From the proof of Theorem \ref{thm:PrecHeavyBall}, for every $\beta\in[0,\beta^\star)$ and every
$\alpha\in(0,\alpha^\star(\beta))$, there exist constants $c_1, c_2 > 0$ such
that~\eqref{eq:descent} holds. We now upper-bound $\Phi_k$ in terms of
$\norm{G_k}^2$ and $\norm{e_k}^2$. By the Polyak--\L ojasiewicz inequality
applied at $\xk$,
\[
  \loss(\xk) - f^\star \leq \frac{1}{\mu}\norm{G_k}^2,
\]
and recalling the definition~\eqref{eq:lyapunov} of $\Phi_k$,
\begin{equation}
  \Phi_k \leq \frac{1}{\mu}\norm{G_k}^2 + A\norm{e_k}^2
         \leq \tilde{C}\bigl(c_1\norm{G_k}^2 + c_2\norm{e_k}^2\bigr),
  \label{eq:phi-upper}
\end{equation}
where $\tilde{C} := \max\{1/(\mu c_1),\, A/c_2\} \in (0,\infty)$.
Combining~\eqref{eq:descent} and~\eqref{eq:phi-upper},
\[
  \Phi_{k+1} - \Phi_k
    \leq -\bigl(c_1\norm{G_k}^2 + c_2\norm{e_k}^2\bigr)
    \leq -\frac{1}{\tilde{C}}\,\Phi_k.
\]
Therefore, defining $\delta := 1/\tilde{C} > 0$, we obtain
$\Phi_{k+1} \leq (1-\delta)\Phi_k$. Since $\Phi_k \geq 0$ for all $k$, we may take $\rho\in(0,1)$ such that $\Phi_{k+1}\leq \rho\Phi_k$.
Iterating gives
\begin{equation*}
    \Phi_k\leq \rho^k\Phi_0.
\end{equation*}
Finally, since $f(x_k)-f^\star\leq \Phi_k$, it follows that
\begin{equation*}
    f(x_k)-f^\star\leq \Phi_0\rho^k.
\end{equation*}
Setting $C:=\Phi_0$ concludes the proof.
\end{proof}

\subsection{Proof of Theorem \ref{thm:muon-compact}}

\paragraph{Statement.} Let $\{\xk,\Hk\}$ be the sequence generated by~\eqref{eq:Muon-Soft} from the initial condition $\{x_0, \dLoss[x_0]\}$. Assume $f$ is proper, $L$-smooth, has a unique global minimizer $x^\star$, and satisfies a global Polyak--\L ojasiewicz inequality with constant $\mu>0$. Then for every $\beta\in[0,1)$ there exists a positive function $\alpha^\star:[0,1)\to\R_{>0}$ such that for every $\alpha\in(0,\alpha^\star(\beta))$ the sequences $\{\xk\}$ and $\{\Hk\}$ are bounded.

\begin{proof}
For every $C\geq \loss(x_0)-f^\star$, fix an open neighbourhood $D$ of $x^\star$
with $\overline{D}\subset\operatorname{int}(D_C)$, where $D_C:=\{f\leq f^\star+C\}$, and
define $r_D:=\operatorname{dist}(\overline{D},D_C^c)>0$,
\begin{equation}
  \gamma(C) := \min_{x\in D_C\setminus D}\|\nabla\!f(x)\|>0,
  \qquad
  \Gamma(C) := \max_{x\in D_C}\|\nabla\!f(x)\|<\infty, \label{eq:gammadefs}
\end{equation}
and
\begin{equation}
  \alpha^\star(\beta,C) :=
    \min\!\left\{
      \frac{2\gamma(C)^2(1-\beta)}{L\bigl(2\beta\Gamma(C)\sqrt{m} + m(1-\beta)\sqrt{\Gamma(C)^2+\eps^2}\bigr)},\;
      r_D(C)/\sqrt{m}
    \right\} > 0. \label{eq:alphastarC}
\end{equation}
Since $\alpha^\star(\beta,C)>0$ for every $C\geq \loss(x_0)-f^\star$, define
$\alpha^\star(\beta):=\sup_{C\geq \loss(x_0)-f^\star}\alpha^\star(\beta,C)>0$.
For any $\alpha\in(0,\alpha^\star(\beta))$, there exists $C\geq \loss(x_0)-f^\star$
such that $\alpha<\alpha^\star(\beta,C)$. Fix such $C$ and write
$\gamma:=\gamma(C)$, $\Gamma:=\Gamma(C)$ for brevity.
 
We prove by induction that $\xk\in D_C$ and $\|\Hk\|\leq\Gamma$ for all
$k\geq 0$. The base case holds since $x_0\in D_C$ by definition of $C$ and
$\|\Hz\|=\|\dLoss[x_0]\|\leq\Gamma$. Assume $x_0,\ldots,\xk\in D_C$
and $\|\Hi{j}\|\leq\Gamma$ for all $j\leq k$. Since $e_0=\Hz-\dLoss[x_0]=0$
and the lag error $e_j:=\Hi{j}-\dLoss[x_j]$ satisfies the recursion
$e_{j+1} = \beta e_j + (\dLoss[x_j]-\dLoss[x_{j+1}])$,
iterating with $\|\dLoss[x_{j+1}]-\dLoss[x_j]\|\leq L\sqrt{m}\alpha$
(from $L$-smoothness and~\eqref{eq:stepbound}) gives
\begin{equation}
  \|e_k\| \leq L\sqrt{m}\alpha\sum_{j=0}^{k-1}\beta^j < \frac{\sqrt{m}L\alpha}{1-\beta},
  \label{eq:ebound}
\end{equation}
uniformly in $k$.
The bound on $\|\Hp\|$ then follows directly:
\begin{equation*}
  \|\Hp\| \leq \beta\|\Hk\|+(1-\beta)\|\dLoss[\xk]\|\leq\Gamma.
\end{equation*}
 
It remains to show $\xp\in D_C$. Since $\xp=\xk-\alpha d_{\eps}(\Hp)$,
we consider two cases. If $\xk\in D$, then
$\|\xp-\xk\|=\alpha\|d_{\eps}(\Hp)\|<\alpha\sqrt{m}\leq r_D$ by~\eqref{eq:stepbound}
and the choice $\alpha<\alpha^\star(\beta,C)\leq r_D$, so $\xp$ cannot exit $D_C$.
If $\xk\in D_C\setminus D$, then $\|\dLoss[\xk]\|\geq\gamma$ and by $L$-smoothness,
\begin{equation*}
  \loss(\xp) \leq \loss(\xk) - \alpha\inner{\dLoss[\xk]}{d_{\eps}(\Hp)}
                          + \frac{L\alpha^2\dimm}{2}.
\end{equation*}
The step $d_{\eps}(\Hp)$ involves $\Hp=\dLoss[\xk]+h_k$, where
the prospective lag $h_k:=\Hp-\dLoss[\xk]=\beta e_k$ satisfies
$\|h_k\|=\beta\|e_k\|<\frac{\beta\sqrt{m}L\alpha}{1-\beta}$ by~\eqref{eq:ebound}.
Using $\|\Hp\|\leq\Gamma$ and $\Hp=\dLoss[\xk]+h_k$, we have
\begin{align*}
  \inner{\dLoss[\xk]}{d_{\eps}(\Hp)}
  &=\inner{\dLoss[\xk]}{(\Hp\Hp^\top+\eps^2 I)^{-1/2}\Hp} \\
  &=\inner{\dLoss[\xk]}{(\Hp\Hp^\top+\eps^2 I)^{-1/2}\dLoss[\xk]}
    \\&\quad+\inner{\dLoss[\xk]}{(\Hp\Hp^\top+\eps^2 I)^{-1/2}h_k} \\
  &\geq\frac{\|\dLoss[\xk]\|^2}{\sqrt{\|\Hp\|^2+\eps^2}}
    -\frac{1}{\eps}\|\dLoss[\xk]\|\|h_k\| \\
  &\geq\frac{\gamma^2}{\sqrt{\Gamma^2+\eps^2}}
    -\frac{\Gamma}{\eps}\|h_k\| \\
  &>\frac{\gamma^2}{\sqrt{\Gamma^2+\eps^2}}
    -\frac{\beta\Gamma\sqrt{m}L\alpha}{\eps(1-\beta)}.
\end{align*}
Substituting,
\begin{equation*}
  \loss(\xp)
    \leq \loss(\xk)
      - \alpha\left(
        \frac{\gamma^2}{\sqrt{\Gamma^2+\eps^2}}
        - \frac{\beta\Gamma\sqrt{m}L\alpha}{\eps(1-\beta)}
      \right)
      + \frac{L\alpha^2 m}{2}
    =: \loss(\xk) - c_\alpha,
\end{equation*}
and $c_\alpha>0$ holds whenever
\begin{equation*}
  \alpha
    < \frac{2\gamma^2(1-\beta)}{L\bigl(2\beta\Gamma\sqrt{m} + m(1-\beta)\sqrt{\Gamma^2+\eps^2}\bigr)},
\end{equation*}
which is precisely the first term in~\eqref{eq:alphastarC}. Therefore
$\loss(\xp)\leq \loss(\xk)\leq f^\star+C$, i.e.\ $\xp\in D_C$.
 
By induction, $\xk\in D_C$ and $\|\Hk\|\leq\Gamma$ for all $k\geq 0$.
Since $f$ is proper, the sublevelset $D_C$ is compact, implying that both sequences are bounded.
\end{proof}

\subsection{Proof of Corollary \ref{cor:muon-convergence}}

\paragraph{Statement.} For the algorithm~\eqref{eq:Muon-Soft}, for every $\beta\in[0,\beta^\star)$
  there exists $\bar\alpha(\beta)>0$ such that for every
  $\alpha\in(0,\bar\alpha(\beta))$ the sequences $\{\xk\}$ and $\{\Hk\}$ are
  bounded and $\lim_{k\to\infty}\|\dLoss[\xk]\|=0$, where
  $\beta^\star=1/\sqrt{2}$ is the threshold from Theorem \ref{thm:PrecHeavyBall}.

\begin{proof}
  Fix any $\beta\in[0,\beta^\star)$ and define
  $\bar\alpha(\beta):=\min\{\alpha_1^\star(\beta),\alpha_2^\star(\beta)\}>0$,
  where $\alpha_1^\star(\beta)$ is as   in Theorem \ref{thm:PrecHeavyBall} and
  $\alpha_2^\star(\beta)$ is the threshold from Theorem \ref{thm:muon-compact}. For any $\alpha\in(0,\bar\alpha(\beta))$, Theorem \ref{thm:muon-compact}
  guarantees that $\{\xk\}$ and $\{\Hk\}$ remain in a compact set, so
  $\|\Hk\|\leq\Gamma$ uniformly and the eigenvalues of
  $\Pk=(\Hk\Hk^\top+\eps^2\Id)^{1/2}$ lie in
  $[{\eps},\sqrt{\Gamma^2+\eps^2}]$ for all $k\geq 0$.
  We then conclude that $\Pk$ is bounded with $\Pmin={\eps}$ and
  $\Pmax=\sqrt{\Gamma^2+\eps^2}$, and Theorem \ref{thm:PrecHeavyBall} gives
  $\lim_{k\to\infty}\|\dLoss[\xk]\|=0$.
\end{proof}

\subsection{Proof of Theorem \ref{thm:precNest}}

\paragraph{Statement.} Consider the preconditioned Nesterov defined in
    \eqref{eq:PrecNesterov-def}, and assume the preconditioning
    matrix $\Pk$ is uniformly bounded -- \ie~there exist $\Pmax>\Pmin>0$ such that $\Pmin I\preceq \Pk\preceq \Pmax I$ for all $k\in\mathbb{N}$.
    Then there exist a constant $\beta^*\in(0,1)$ and a positive
    function $\alpha^*:[0,\beta^*)\to\re_{>0}$ such that for every
    $\beta\in[0,\beta^*)$ and every $\alpha\in(0,\alpha^*(\beta))$,
    it holds that
    \begin{equation*}
        \lim_{k\to\infty}\|\nabla \loss(\yk)\|=0.
    \end{equation*} 

\begin{proof}
    Let $\lmin:=\inf_{\x\in\PSp}\loss(\x)$ and define the function
    \begin{equation}
        \Phi(\xk):=\loss(\xk)-\lmin+A\|\xk-\xm\|^2.
    \end{equation}
    We will prove convergence by showing that $\Phi(\cdot)$ is a
    Lyapunov function of the discrete time system
    \eqref{eq:PrecNesterov-def}. Henceforth consider the shorthand
    $\Phi_k:=\Phi(\xk)$, $\Gk:=\dLoss[\yk]$, and
    $d_k:=\Pk^{-1}\Gk$. Define further $s_k:=\xk-\xm$ and then
    rewrite \eqref{eq:PrecNesterov-def} as
    %
    \begin{subequations}
        \label{eq:thm1proofsystem}
        \begin{align}
            &\xp=\xk+\beta s_k-\alpha d_k, \\
            &s_{k+1}=\beta s_k-\alpha d_k.
        \end{align}
    \end{subequations}
    To show that $\Phi_k$ is a Lyapunov function, we need to show
    that it decreases at every iteration, \ie~we need to show that
    $\Phi_{k+1}-\Phi_k<0$. Consider
    \begin{align*}
        \Phi_{k+1}-\Phi_k
        =
        \loss(\xp)-\loss(\xk)+A(\|s_{k+1}\|^2-\|s_k\|^2).
    \end{align*}

    We will first look at the term $\|s_{k+1}\|^2-\|s_k\|^2$.
    Note that
    \begin{align*}
        \|s_{k+1}\|^2
        &=
        \|\beta s_k-\alpha d_k\|^2 \\
        &\leq
        \beta^2\|s_k\|^2+\alpha^2\|d_k\|^2+2\alpha\beta\|s_k\|\|d_k\|.
    \end{align*}

    Next, consider $\loss(\xp)-\loss(\xk)$, and note that from
    $L$-smoothness of $\loss$ and uniform boundedness of $\Pk$ we
    have that
    %
    \begin{align*}
        \loss(\xp)
        &\leq
        \loss(\yk)-\alpha\ip[\Gk,d_k]+\frac{L\alpha^2}{2}\|d_k\|^2 \\
        &\leq
        \loss(\yk)-\underbrace{\left(\alpha \Pmin-\frac{\alpha^2L}{2}
        \right)}_{c_\alpha}\|d_k\|^2,
    \end{align*}
    and note that $c_\alpha>0$ for all $0<\alpha<2\Pmin/L$. Next, and
    again by $L$-smoothness of $\loss$, using the fact that
    $\yk=\xk+\beta s_k$, and Young's inequality we get that
    \begin{align*}
        \loss(\yk)
        &\leq
        \loss(\xk)+\beta\ip[\dLoss[\xk],s_k]
        +\frac{L}{2}\beta^2\|s_k\|^2.
    \end{align*}

    Then, note that $\dLoss[\xk]=\Gk+e_k$ where
    $e_k=\dLoss[\xk]-\dLoss[\yk]$. Then by $L$-smoothness
    of $\loss$ and uniform boundedness of $\Pk$
    \begin{align*}
        \ip[\dLoss[\xk],s_k]
        &=
        \ip[\Gk+e_k,s_k] \\
        &=
        \ip[\Gk,s_k]+\ip[e_k,s_k] \\
        &\leq
        \Pmax\|d_k\|\|s_k\|+L\beta\|s_k\|^2.
    \end{align*}

    Hence
    \begin{align*}
        \loss(\yk)
        &\leq
        \loss(\xk)+\beta \Pmax\|d_k\|\|s_k\|
        +\left(L\beta^2+\frac{L}{2}\beta^2\right)\|s_k\|^2\\
        &=\loss(\xk)+\beta \Pmax\|d_k\|\|s_k\|
        +\frac{3L}{2}\beta^2\|s_k\|^2.
    \end{align*}

    Putting the current equations together results in
    \begin{equation*}
        \loss(\xp)
        \leq
        \loss(\xk)+\frac{3L}{2}\beta^2\|s_k\|^2
        +\beta \Pmax\|d_k\|\|s_k\|-c_\alpha\|d_k\|^2.
    \end{equation*}

    Writing the iteration for $\Phi_k$ then results
    in
    %
    \begin{align*}
        \Phi_{k+1}-\Phi_k
        &\leq
        \underbrace{\left(\frac{3L}{2}\beta^2+A(\beta^2-1)
        \right)}_{=:\,-C_s(\alpha,\beta)}\|s_k\|^2
        -\underbrace{\left(c_\alpha-A\alpha^2\right)}_{=:\,C_d(\alpha)}
        \|d_k\|^2
        +\underbrace{\beta(\Pmax+2A\alpha)}_{=:\,C_c(\alpha,\beta)}
        \|s_k\|\|d_k\|.
    \end{align*}

    Now choose
    \begin{equation*}
        A=\frac{\Pmin}{4\alpha}.
    \end{equation*}
    Then $2A\alpha=\Pmin/2$, so $C_c(\alpha,\beta)=\beta(\Pmax+\Pmin/2)$ is
    independent of $\alpha$, and
    \begin{equation*}
        C_d(\alpha)
        =
        c_\alpha-A\alpha^2
        =
        \alpha\left(\frac{3\Pmin}{4}-\frac{L\alpha}{2}\right)
        >0
        \qquad\text{for all }\alpha\in\left(0,\frac{3\Pmin}{2L}\right)=:(0,\alpha_1).
    \end{equation*}
    Hence $\Phi_{k+1}-\Phi_k\leq0$ whenever the quadratic form
    $C_s\|s_k\|^2+C_d\|d_k\|^2-C_c\|s_k\|\|d_k\|$ is positive
    definite, which by Sylvester's criterion holds if and only if
    $C_s>0$ and $4C_sC_d-C_c^2>0$.

    For fixed $\beta\in[0,1)$, the condition $C_s>0$ reads
    explicitly
    \begin{equation}
        \label{eq:Cs-pos}
        \alpha < \frac{\Pmin(1-\beta^2)}{6L\beta^2} =: \alpha_2(\beta),
    \end{equation}
    with no constraint for $\beta=0$. Substituting
    $C_c=\beta(\Pmax+\Pmin/2)$, $C_s=\Pmin(1-\beta^2)/(4\alpha)-3L\beta^2/2$,
    and $C_d=\alpha(3\Pmin/4-L\alpha/2)$ into $4C_sC_d-C_c^2>0$ and
    expanding, the discriminant condition is equivalent to the
    quadratic inequality in $\alpha$
    \begin{equation}
        \label{eq:quad-alpha}
        3L^2\beta^2\,\alpha^2
        -\frac{L\Pmin(1+8\beta^2)}{2}\,\alpha
        +\frac{3\Pmin^2(1-\beta^2)}{4}
        -\beta^2\!\left(\Pmax+\frac{\Pmin}{2}\right)^{\!2}
        > 0.
    \end{equation}
    The constant term in \eqref{eq:quad-alpha} is strictly positive
    if and only if
    \begin{equation*}
        \beta^2
        <
        \frac{3\Pmin^2/4}{3\Pmin^2/4+\left(\Pmax+\Pmin/2\right)^2},
    \end{equation*}
    i.e.\ $\beta<\beta^*$ where
    \begin{equation}
        \label{eq:betastar-thm1}
        \beta^*
        :=
        \frac{\sqrt{3}\Pmin/2}{\sqrt{3\Pmin^2/4+\left(\Pmax+\Pmin/2\right)^2}}
        \;\in\;(0,1).
    \end{equation}
    Since the leading coefficient of \eqref{eq:quad-alpha} is
    positive, the linear coefficient is negative, and the constant
    term is positive for $\beta<\beta^*$, the quadratic opens upward
    and is positive at $\alpha=0$, so it has two positive roots
    $\alpha_-(\beta)\leq\alpha_+(\beta)$ and is positive on
    $(0,\alpha_-(\beta))$. Define
    \begin{equation}
        \label{eq:Delta}
        \Delta(\beta)
        :=
        \frac{L^2\Pmin^2(1+8\beta^2)^2}{4}
        -3L^2\beta^2\!\left[
            \frac{3\Pmin^2(1-\beta^2)}{4}
            -\beta^2\!\left(\Pmax+\frac{\Pmin}{2}\right)^{\!2}
        \right],
    \end{equation}
    so that
    \begin{equation}
        \label{eq:alphastar-explicit}
        \alpha_-(\beta)
        =
        \frac{
            L\Pmin(1+8\beta^2)/2 - \sqrt{\Delta(\beta)}
        }{
            6L^2\beta^2
        },
    \end{equation}
    and note that if $\beta=0$ then the problem reduces to a preconditioned gradient descent which can easily be shown satisfy this Theorem.  To see that $\alpha_-(\beta)<\alpha_1$, evaluate
    \eqref{eq:quad-alpha} at $\alpha=\alpha_1=3\Pmin/(2L)$:
    \begin{equation*}
        3L^2\beta^2\cdot\frac{9\Pmin^2}{4L^2}
        -\frac{L\Pmin(1+8\beta^2)}{2}\cdot\frac{3\Pmin}{2L}
        +\frac{3\Pmin^2(1-\beta^2)}{4}
        -\beta^2\!\left(\Pmax+\frac{\Pmin}{2}\right)^{\!2}
        =
        -\beta^2\!\left(\Pmax+\frac{\Pmin}{2}\right)^{\!2}
        \leq 0,
    \end{equation*}
    so the quadratic is non-positive at $\alpha_1$, hence
    $\alpha_-(\beta)<\alpha_1$ for all $\beta\in(0,\beta^*)$, and
    $C_d(\alpha)>0$ is automatically satisfied on
    $(0,\alpha_-(\beta))\subset(0,\alpha_1)$.
    Moreover, $\alpha_-(\beta)<\alpha_2(\beta)$ always holds for
    $\beta<\beta^*$, since at $\alpha=\alpha_2(\beta)$ one has
    $C_s=0$, so $4C_sC_d-C_c^2=-C_c^2<0$, meaning $\alpha_2(\beta)$
    lies above the root $\alpha_-(\beta)$, and hence $C_s>0$ is
    automatically satisfied on $(0,\alpha_-(\beta))$.
    Setting $\alpha^*(\beta):=\alpha_-(\beta)$, all three conditions
    $C_s>0$, $C_d>0$, and $4C_sC_d>C_c^2$ hold for every
    $\alpha\in(0,\alpha^*(\beta))$.

    It remains to convert positive definiteness of the quadratic
    form into a bound with independent squared terms. Applying
    Young's inequality with $\delta=\sqrt{C_d/C_s}>0$ to the cross term gives
    \begin{equation*}
        C_c\|s_k\|\|d_k\|
        \leq
        \frac{C_c\sqrt{C_s}}{2\sqrt{C_d}}\|s_k\|^2
        +\frac{C_c\sqrt{C_d}}{2\sqrt{C_s}}\|d_k\|^2,
    \end{equation*}
    and therefore
    \begin{align*}
        C_s\|s_k\|^2+C_d\|d_k\|^2-C_c\|s_k\|\|d_k\|
        \geq
        \underbrace{\sqrt{C_s}\!\left(\sqrt{C_s}
        -\frac{C_c}{2\sqrt{C_d}}\right)}_{=:\,c_1}
        \|s_k\|^2
        +
        \underbrace{\sqrt{C_d}\!\left(\sqrt{C_d}
        -\frac{C_c}{2\sqrt{C_s}}\right)}_{=:\,c_2}
        \|d_k\|^2,
    \end{align*}
    where $c_1,c_2>0$ since $4C_sC_d>C_c^2$ implies
    $2\sqrt{C_sC_d}>C_c$. Therefore
    \begin{equation*}
        \Phi_{k+1}-\Phi_k
        \leq
        -c_1\|s_k\|^2-c_2\|d_k\|^2.
    \end{equation*}

    Since $\Phi_k\ge0$ by definition, summing over $k$ gives
    \begin{equation*}
        \sum_{k=0}^\infty \|s_k\|^2<\infty,
        \qquad
        \sum_{k=0}^\infty \|d_k\|^2<\infty.
    \end{equation*}

    In particular, $\lim_{k\to\infty}\|d_k\|=0$. Finally, by
    uniform boundedness of $\Pk$,
    \begin{equation*}
        \|\Gk\|=\|\Pk d_k\|\le \Pmax\|d_k\|\to0,
    \end{equation*}
    which proves \eqref{eq:PrecNestConvThm}.
\end{proof}

\subsection{Proof of Corollary \ref{cor:precNest-PL}}

\paragraph{Statement.} Let $\loss$ satisfy a global Polyak--\L{}ojasiewicz inequality with constant $\mu>0$.
    Then, for every $\beta\in[0,\beta^\star)$ and every
    $\alpha\in(0,\alpha^\star(\beta))$ (where $\beta^*$ and
    $\alpha^*(\cdot)$ are as in Theorem~\ref{thm:precNest}), there exist constants
    $C>0$ and $\rho\in(0,1)$ such that
    \begin{equation}
        f(x_k)-f^\star \leq C\rho^k,
        \qquad \forall k\in\mathbb N.
    \end{equation}

\begin{proof}
    From the proof of Theorem~\ref{thm:precNest}, for every
    $\beta\in[0,\beta^*)$ and every $\alpha\in(0,\alpha^*(\beta))$,
    there exist constants $c_1,c_2>0$ such that
    \begin{equation}
        \label{eq:lyap-dec-cor}
        \Phi_{k+1}-\Phi_k\leq -c_1\|s_k\|^2-c_2\|d_k\|^2.
    \end{equation}

    We now upper bound $\Phi_k$ in terms of $\|s_k\|^2$ and
    $\|d_k\|^2$. By $L$-smoothness of $\loss$ and the identity
    $\yk=\xk+\beta s_k$, we have
    \begin{align*}
        \loss(\xk)
        &\leq
        \loss(\yk)+\ip[\Gk,\xk-\yk]+\frac{L}{2}\|\xk-\yk\|^2 \\
        &=
        \loss(\yk)-\beta\ip[\Gk,s_k]+\frac{L}{2}\beta^2\|s_k\|^2 \\
        &\leq
        \loss(\yk)+\beta\|\Gk\|\|s_k\|+\frac{L}{2}\beta^2\|s_k\|^2.
    \end{align*}
    Applying the P\L{} inequality at $\x=\yk$ gives
    \begin{equation*}
        \loss(\yk)-\lmin\leq \frac{1}{\mu}\|\Gk\|^2.
    \end{equation*}
    Therefore,
    \begin{equation*}
        \loss(\xk)-\lmin
        \leq
        \frac{1}{\mu}\|\Gk\|^2+\beta\|\Gk\|\|s_k\|+\frac{L}{2}\beta^2\|s_k\|^2.
    \end{equation*}

    Next, by uniform boundedness of $\Pk$ we have
    $\|\Gk\|=\|\Pk d_k\|\leq \Pmax\|d_k\|$, hence
    \begin{equation*}
        \loss(\xk)-\lmin
        \leq
        \frac{\Pmax^2}{\mu}\|d_k\|^2+\beta \Pmax\|d_k\|\|s_k\|+\frac{L}{2}\beta^2\|s_k\|^2.
    \end{equation*}
    Applying Young's inequality with parameter $\eta>0$ to the
    cross term,
    \begin{equation*}
        \beta \Pmax\|d_k\|\|s_k\|
        \leq
        \frac{\eta}{2}\|d_k\|^2+\frac{\beta^2\Pmax^2}{2\eta}\|s_k\|^2,
    \end{equation*}
    gives
    \begin{equation*}
        \loss(\xk)-\lmin
        \leq
        \left(\frac{\Pmax^2}{\mu}+\frac{\eta}{2}\right)\|d_k\|^2
        +
        \left(\frac{\beta^2\Pmax^2}{2\eta}+\frac{L}{2}\beta^2\right)\|s_k\|^2.
    \end{equation*}

    Recalling the definition of $\Phi_k$, we conclude that
    \begin{equation}
        \label{eq:phi-upper-cor}
        \Phi_k
        \leq
        \underbrace{\left(\frac{\Pmax^2}{\mu}+\frac{\eta}{2}\right)}_{=:\,C_d}
        \|d_k\|^2
        +
        \underbrace{\left(A+\frac{\beta^2\Pmax^2}{2\eta}
        +\frac{L}{2}\beta^2\right)}_{=:\,C_s}
        \|s_k\|^2.
    \end{equation}

    Combining \eqref{eq:phi-upper-cor} with \eqref{eq:lyap-dec-cor}
    gives
    \begin{align*}
        \Phi_{k+1}-\Phi_k
        &\leq
        -c_1\|s_k\|^2-c_2\|d_k\|^2 \\
        &\leq
        -\min\!\left\{\frac{c_2}{C_d},\frac{c_1}{C_s}\right\}
        \!\left(C_d\|d_k\|^2+C_s\|s_k\|^2\right) \\
        &\leq
        -\min\!\left\{\frac{c_2}{C_d},\frac{c_1}{C_s}\right\}\Phi_k.
    \end{align*}
    Therefore, defining
    $\delta:=\min\{c_2/C_d,\,c_1/C_s\}>0$,
    we obtain $\Phi_{k+1}\leq(1-\delta)\Phi_k$. Since $\Phi_k\geq0$ for all $k$, we may take $\rho\in(0,1)$ such that $\Phi_{k+1}\leq\rho\Phi_k$. Iterating gives
    \[
        \Phi_k\leq \rho^k\Phi_0.
    \]
    Finally, since $f(x_k)-f^\star\leq\Phi_k$, it follows that
    \[
        f(x_k)-f^\star\leq \Phi_0\rho^k.
    \]
    Setting $C:=\Phi_0$ concludes the proof.
\end{proof}

\subsection{Proof of Theorem \ref{thm:Muesterov-Precompact}}

\paragraph{Statement.} Let $\{\xk,\yk\}$ be the sequence generated by \eqref{eq:Muesterov-def} from the initial condition $\{x_0,x_0\}$. Assume $f$ is proper, $L$-smooth, has a unique global minimiser $x^*$, and satisfies a global P\L{} inequality with constant $\mu>0$. Then there exist $\beta^*\in(0,1]$ and a positive function $\alpha^*:[0,\beta^*)\to\re_{>0}$ such that for every $\beta\in[0,\beta^*)$ and every $\alpha\in(0,\alpha^*(\beta))$ the sequences $\{\xk\}$ and $\{\yk\}$ are bounded.

\begin{proof}
    Throughout this proof, $\|\cdot\|$ and $\langle\cdot,\cdot\rangle$
    denote the Frobenius norm and inner product on $\re[\n\times \m]$,
    i.e.\ $\|A\|^2=\operatorname{tr}(A^\top A)$ and
    $\langle A,B\rangle=\operatorname{tr}(A^\top B)$.

    For every $C\geq \loss(x_0)-f^*$, fix an open neighborhood $\Dset$
    of $x^*$ with $\overline{\Dset}\subset\operatorname{int}(\DCset)$
    where $\DCset:=\{f\leq f^*+C\}$, and define
    $\Kset:=\DCset\setminus{\Dset}$,
    $r_{\Dset}:=\operatorname{dist}(\overline{\Dset},\DCset^c)>0$,
    \begin{equation}
        \label{eq:gamma-Gamma}
        \gamma(C) := \min_{x\in \Kset}\|\nabla f(x)\| > 0,
        \qquad
        \Gamma(C) := \max_{x\in \DCset}\|\nabla f(x)\| < \infty,
    \end{equation}
    and
    \begin{equation}
        \label{eq:beta-star-C}
        \beta^*(C)
        :=
        \frac{
            \gamma(C)^2/\!\left(\Gamma(C)\sqrt{\m}\sqrt{\gamma(C)^2+\eps^2}\right)
        }{
            1+\gamma(C)^2/\!\left(\Gamma(C)\sqrt{\m}\sqrt{\gamma(C)^2+\eps^2}\right)
        }
        \;\in\;(0,1).
    \end{equation}
    Since $\beta^*(C)\in(0,1)$ for every $C\geq \loss(x_0)-f^*$, the
    supremum
    \begin{equation}
        \label{eq:beta-star}
        \beta^* := \sup_{C\,\geq\, \loss(x_0)-f^*} \beta^*(C)
    \end{equation}
    is well-defined and satisfies $\beta^*\in(0,1]$.
    Now fix any $\beta\in[0,\beta^*)$ and define the set of
    admissible levels
    \begin{equation}
        \label{eq:Cadmissible}
        \mathcal{C}(\beta)
        :=
        \bigl\{C\geq \loss(x_0)-f^* \;:\; \beta < \beta^*(C)\bigr\},
    \end{equation}
    which is non-empty by definition of $\beta^*$: since
    $\beta<\beta^*=\sup_C\beta^*(C)$, there exists at least one
    $C\in\mathcal{C}(\beta)$.
    The remainder of the proof holds for every $C\in\mathcal{C}(\beta)$;
    we write $\gamma:=\gamma(C)$ and $\Gamma:=\Gamma(C)$ for brevity.

    Write $s_k:=\xk-\xm$ and
    $d_{\eps}(y):=(\nabla f(y)\nabla f(y)^\top+\eps^2 I)^{-1/2}
    \nabla f(y)$.
    If $\nabla f(y)=U\Sigma V^\top$ is the SVD with singular values
    $\sigma_1\geq\cdots\geq\sigma_m\geq 0$, then
    $d_{\eps}(y)=U(\Sigma^2+\eps^2 I)^{-1/2}\Sigma V^\top$
    with singular values $\sigma_i/\sqrt{\sigma_i^2+\eps^2}<1$,
    so
    \begin{equation}
        \label{eq:dbound}
        \|d_{\eps}(y)\|^2
        = \sum_{i=1}^{\m} \frac{\sigma_i^2}{\sigma_i^2+\eps^2}
        < \m.
    \end{equation}
    The recursion $s_{k+1}=\beta s_k-\alpha d_{\eps}(\yk)$
    with $s_0=0$ and \eqref{eq:dbound} give, by induction,
    \begin{equation}
        \label{eq:vel}
        \|s_k\|\leq\frac{\alpha\sqrt{\m}}{1-\beta}=:\bar{s},
        \qquad
        \|\yk-\xk\|\leq\rho:=\frac{\alpha\beta\sqrt{\m}}{1-\beta},
        \qquad\forall\,k\geq 0,
    \end{equation}
    unconditionally for any $\beta\in[0,1)$.

    For $y\in \Kset$, by direct computation using the SVD,
    \begin{equation}
        \label{eq:inner-exact}
        \langle\nabla f(y),d_{\eps}(y)\rangle
        = \sum_{i=1}^{\m}\frac{\sigma_i^2}{\sqrt{\sigma_i^2+\eps^2}}.
    \end{equation}
    The function $\phi(t):=t/\sqrt{t+\eps^2}$ is concave on
    $[0,\infty)$ with $\phi(0)=0$, hence subadditive, so
    \begin{equation}
        \label{eq:subadditive}
        \sum_{i=1}^{\m}\phi(\sigma_i^2)
        \geq \phi\!\left(\sum_{i=1}^{\m}\sigma_i^2\right)
        = \frac{\|\nabla f(y)\|^2}{\sqrt{\|\nabla f(y)\|^2+\eps^2}}
        \geq
        \frac{\gamma^2}{\sqrt{\gamma^2+\eps^2}},
    \end{equation}
    where the last inequality uses that $t\mapsto t^2/\sqrt{t^2+\eps^2}$
    is increasing and $\|\nabla f(y)\|\geq\gamma$ on $\Kset$.
    By $L$-smoothness and \eqref{eq:dbound},
    \begin{align}
        f(y-\alpha d_{\eps}(y))
        &\leq f(y)
            - \alpha\langle\nabla f(y), d_{\eps}(y)\rangle
            + \frac{L}{2}\alpha^2\|d_{\eps}(y)\|^2 \nonumber\\
        &\leq f(y)
            - \left(\frac{\alpha\gamma^2}{\sqrt{\gamma^2+\eps^2}}
            - \frac{L\m}{2}\alpha^2\right)
         \nonumber\\&=: f(y) - c_\alpha,\label{eq:fdec}
    \end{align}
    where $c_\alpha>0$ for all $\alpha\in(0,\alpha_1)$ with
    $\alpha_1(C):=2\gamma^2/(L\m\sqrt{\gamma^2+\eps^2})>0$.

    Next, we show $\yk\in\DCset$ implies $\xp,\yp\in\DCset$. We consider two cases.

    \emph{Case 1: $\yk\in\Dset$.}
    Since $\|\xp-\yk\|=\|\alpha d_{\eps}(\yk)\|\leq\alpha\sqrt{\m}$ and
    $\|\yp-\yk\|\leq\|\yp-\xp\|+\|\xp-\yk\|\leq\rho+\alpha\sqrt{\m}=\bar{s}$, both $\xp$ and
    $\yp$ remain in $\DCset$ provided $\bar{s}<r_{\Dset}$, i.e.
    \begin{equation}
        \label{eq:cond1}
        \frac{\alpha\sqrt{\m}}{1-\beta} < r_{\Dset}(C).
    \end{equation}

    \emph{Case 2: $\yk\in\Kset$.}
    By \eqref{eq:fdec}, $\loss(\xp)\leq f(\yk)-c_\alpha\leq f^*+C$,
    so $\xp\in\DCset$ and $\|\nabla \loss(\xp)\|\leq\Gamma$.
    By $L$-smoothness,
    $f(\yp)\leq \loss(\xp)+\Gamma\rho+\frac{L}{2}\rho^2$.
    For $\yp\in\DCset$ it suffices to have
    $\Gamma\rho+\frac{L}{2}\rho^2<c_\alpha$, since then
    $f(\yp)<\loss(\xp)+c_\alpha\leq f(\yk)\leq f^*+C$, where the last inequality in the chain comes from $\yk\in\Kset$.
    We therefore require
    \begin{equation}
        \label{eq:cond2}
        \Gamma\rho+\frac{L}{2}\rho^2 < c_\alpha.
    \end{equation}

    We next look into admissible ranges for $\alpha$. Condition \eqref{eq:cond1} holds for all
    $\alpha<(1-\beta)r_{\Dset}(C)/\sqrt{\m}$.
    Substituting $\rho=\alpha\beta\sqrt{\m}/(1-\beta)$ into
    \eqref{eq:cond2} and rearranging, the condition holds whenever
    $\alpha<\alpha_2(\beta,C)$ where
    \begin{equation}
        \label{eq:alpha2}
        \alpha_2(\beta,C)
        :=
        \frac{2(1-\beta)^2}{L\m(1-2\beta+2\beta^2)}
        \left(
            \frac{\gamma^2}{\sqrt{\gamma^2+\eps^2}}
            - \frac{\Gamma\beta\sqrt{\m}}{1-\beta}
        \right).
    \end{equation}
    Since $C\in\mathcal{C}(\beta)$, we have $\beta<\beta^*(C)$,
    which is precisely the condition that makes the bracket in
    \eqref{eq:alpha2} strictly positive, so $\alpha_2(\beta,C)>0$.
    Defining
    \begin{equation}
        \label{eq:alphastar-C}
        \alpha^*(\beta,C)
        :=\min\!\left\{
            \alpha_1(C),\;
            \alpha_2(\beta,C),\;
            \frac{(1-\beta)r_{\Dset}(C)}{\sqrt{\m}}
        \right\}>0,
    \end{equation}
    all three terms of which are strictly positive for every
    $C\in\mathcal{C}(\beta)$, both \eqref{eq:cond1} and
    \eqref{eq:cond2} hold for all $\alpha\in(0,\alpha^*(\beta,C))$.
    Since $\alpha^*(\beta,C)>0$ for every $C\in\mathcal{C}(\beta)$
    and $\mathcal{C}(\beta)$ is non-empty, define
    \begin{equation}
        \label{eq:alphastar}
        \alpha^*(\beta)
        :=
        \sup_{C\,\in\,\mathcal{C}(\beta)}\alpha^*(\beta,C)\;>\;0.
    \end{equation}
    For any $\alpha\in(0,\alpha^*(\beta))$, by definition of the
    supremum there exists $C\in\mathcal{C}(\beta)$ such that
    $\alpha<\alpha^*(\beta,C)$, and hence both \eqref{eq:cond1}
    and \eqref{eq:cond2} are satisfied for this $C$.

    Since $C\geq \loss(x_0)-f^*$, we have $x_0,y_0\in\DCset$.
    The above shows $\yk\in\DCset$ implies $\xp,\yp\in\DCset$ for
    all $\alpha\in(0,\alpha^*(\beta))$ and $\beta\in[0,\beta^*)$.
    By induction, $\xk,\yk\in\DCset$ for all $k\geq 0$.
    Since $\DCset$ is compact, the sequences are bounded.
\end{proof}

\subsection{Proof of Corollary \ref{cor:Muesterov-Convergence}}

\paragraph{Statement.} For the algorithm presented in \eqref{eq:Muesterov-def}, there exists $\bar\beta\in(0,1)$ and $\bar\alpha:[0,\bar\beta)\to\re_+$, such that every $\beta\in[0,\bar\beta)$ and every $\alpha\in(0,\bar\alpha(\beta))$ the sequences $\{\xk\}$ and $\{\yk\}$ are bounded and $\lim_{k\to\infty}\|\dLoss[\yk]\|=0$.

\begin{proof}
    Define $\bar\beta:=\min\{\beta^*_3,\beta^*_4\}
    \in(0,1)$ and $\bar\alpha(\beta):=\min\{\alpha^*_3(\beta),
    \alpha^*_4(\beta)\}>0$, where $\beta^*_i$ and $\alpha^*_i(\cdot)$
    are as in Theorem~$i$ for $i=3,4$. Fix any $\beta\in[0,\bar\beta)$ and $\alpha\in(0,\bar\alpha(\beta))$.
    Since $\beta<\beta^*_4$ and $\alpha<\alpha^*_4(\beta)$,
    Theorem~\ref{thm:Muesterov-Precompact} guarantees that $\{\xk\}$ and
    $\{\yk\}$ remain in the compact set $\DCset$ for some
    $C\in\mathcal{C}(\beta)$.
    In particular, the eigenvalues of
    $\Pk=(\dLoss[\yk]\dLoss[\yk]^\top+\eps^2 I)^{1/2}$
    lie in $[\eps,\sqrt{\Gamma(C)^2+\eps^2}]$
    for all $k\geq 0$, $\Pk$ is bounded with
    $\Pmin=\eps$ and $\Pmax=\sqrt{\Gamma(C)^2+\eps^2}$.
    Since $\beta<\beta^*_3$ and $\alpha<\alpha^*_3(\beta)$,
    Theorem~\ref{thm:precNest} then gives
    $\lim_{k\to\infty}\|\dLoss[\yk]\|=0$.
\end{proof}

\newpage 

\end{document}